\documentclass[a4paper, 11pt]{amsart}
\usepackage{dsfont}
\usepackage{amssymb}
\usepackage{amsfonts}
\usepackage{amsthm}
\usepackage{amsmath}
\usepackage{verbatim}
\usepackage{a4wide}
\usepackage{bm}
\usepackage{color}
\usepackage{appendix}
\usepackage{amssymb}
\usepackage{bbm}
\usepackage{fancybox}
\usepackage{fancyhdr}
\usepackage{graphicx}
\usepackage{subfigure}

\newtheorem{lemma}{Lemma}[section]

\theoremstyle{definition}

\theoremstyle{definition}
\newtheorem{remark}[lemma]{Remark}
\theoremstyle{definition}

{\catcode `\@=11 \global\let\AddToReset=\@addtoreset}
\AddToReset{equation}{section}

\newcommand{\vphi}{{\varphi }}
\newcommand{\veps}{{\varepsilon }}

\newcommand{\bv}{\mathbf v}

\newcommand{\iij}{{i,j}}

\newcommand{\ipj}{{i+1,j}}
\newcommand{\ijp}{{i,j+1}}

\newcommand{\imj}{{i-1,j}}
\newcommand{\ijm}{{i,j-1}}

\newcommand{\nij}{{0,i,j}}

\newcommand{\nipj}{{0,i+1,j}}
\newcommand{\nijp}{{0,i,j+1}}

\newcommand{\nimj}{{0,i-1,j}}
\newcommand{\nijm}{{0,i,j-1}}

\newcommand{\np}{{n+1}}

\newcommand{\R}{\mathbb{R}}

\newcommand{\diver}{{\rm div}}

\newcommand{\bnu}{\bm \nu}

\newcommand{\del}{\partial}

\begin{document}

\title[Asymptotic Preserving Scheme for the Euler-Korteweg Model]
{Low Mach Asymptotic Preserving Scheme for the Euler-Korteweg Model}

\author[J.~Giesselmann] {Jan Giesselmann}

\address{Weierstrass Institute for applied analysis and Stochastics, Mohrenstr. 39, 10117 Berlin, Germany}\email{jan.giesselmann@wias-berlin.de}

\subjclass[2010]{65M06, 65M12, 76T10}

\keywords{Multi-phase flows, phase transition, all-speed scheme, asymptotic preserving, low Mach number flow, finite difference scheme
 }

\thanks{The author would like to thank Christian Rohde (University of Stuttgart) for suggesting this topic. This work was supported by the EU FP7-REGPOT project 'Archimedes Center for
Modeling, Analysis and Computation' and by the German Research Foundation (DFG) via financial support of the project
`Modeling and sharp interface limits of local and non-local generalized
  Navier-Stokes-Korteweg Systems'.}

\begin{abstract}
We present an all speed scheme for the Euler-Korteweg model.
We study a semi-implicit time-discretisation which treats the  terms, which are stiff for low Mach numbers, implicitly and thereby avoids a dependence of the timestep restriction on the Mach number.
Based on this we present a fully discrete finite difference scheme.
In particular, the scheme is asymptotic preserving, i.e., it converges to a stable discretisation of the incompressible limit 
of the Euler-Korteweg model when the Mach number tends to zero.
\end{abstract}

\maketitle
\thispagestyle{empty}

\section{Introduction}
This work is concerned with the numerical simulation of compressible multi-phase flows 
 via a phase field approach. Specifically we consider the isothermal Euler-Korteweg (EK) model, see e.g. 
 \cite{BDD07, BDDJ07}, which allows for mass fluxes across the interface.
This model is a diffuse interface model solving one set of partial differential equations (PDEs) on the whole computational domain.
The solution of this PDE system already contains the position of the phase boundary.
There are several works on numerical methods for the Euler-Korteweg and the Navier-Stokes-Korteweg equations, see \cite{Die07,BP13,GMP13}. 
 All these works focus on stability properties of their respective schemes. 
In particular, in \cite{BP13,GMP13}, fully implicit time discretisations are used to obtain stable schemes.
We aim at constructing a scheme which is computationally faster as it  needs to solve only one implicit equation per timestep 
while still being stable for reasonable timestep sizes independent of the 
Mach number.

For background on asymptotic preserving and all-speed schemes let us refer to the further developed case of single phase flows
while we like to stress that all speed schemes are of particular importance  in multi-phase flows due to the different speeds of sound in both phases.
However, unsteady compressible flows with small or strongly varying Mach number occur in many physical and engineering applications, 
and are not limited to multi-phase phenomena.

While the general development of shock capturing schemes for compressible flows is quite mature, these schemes
encounter severe restrictions in case of low Mach flows. 
These problems are due to the speed of acoustic waves being much larger than the speed of the flow. 
In fact, explicit-in-time shock capturing schemes need to satisfy a Courant-Friedrichs-Levy (CFL) timestep restriction in order to be stable.
This condition states that the maximal timestep is inversely proportional to the maximal wave speed which scales with the reciprocal of the Mach number.
In addition, these schemes also need artificial dissipation proportional to the maximal wave speed. Therefore, for small Mach numbers, 
 the spatial resolution has to be 
very high to ensure that the solution is not dominated by artificial viscosity.
There have been many contributions concerning \emph{all speed schemes} for compressible flows, i.e., schemes which work well for space and time discretisations
independent of the Mach number. Different approaches for the
isentropic Euler equations can be found in,  \cite[e.g.]{DT11,HJL12,PM05,VVW03}.

The scheme at hand is based in the asymptotic preserving (AP) methodology. For a family of models $M^\veps$ converging to a limit model  
$M^0$ for $\veps \rightarrow 0$
this methodology 
consists in constructing discretisations $M^\veps_{\Delta}$ of $M^\veps$ such that for fixed discretisation parameter $\Delta$ the limit
$\lim_{\veps \rightarrow 0} M^\veps_\Delta$ is a stable and consistent discretisation of $M^0.$
Since the fundamental works \cite{Jin95,JL96}  asymptotic preserving schemes have been the topic of many studies in computational fluid dynamics, 
in recent years, see \cite{AABCP11,BLT13,BD06,CCGPS10,DET08} and references therein. In particular, the algorithm presented here is inspired by \cite{DT11}.

To be more precise let us introduce the model under consideration:
On some space--time domain $\Omega \times (0,T)$ with $T>0$ and $\Omega \subset \R^d$ open and bounded with Lipschitz--boundary
 we study the following balance laws for the density $\rho$ and the velocity $\bv:$
\begin{equation}\label{eq:EK}
 \begin{split}
  \rho_t + \diver (\rho \bv) &=0\\
(\rho \bv)_t + \diver( \rho \bv \otimes \bv) + \frac{1}{M^2}\nabla p(\rho) &= 
\frac{\gamma}{M^2} \diver \Big(\Big(\rho \Delta \rho + \frac{1}{2} |\nabla \rho|^2\Big) {\bf I} - \nabla \rho \otimes \nabla \rho \Big)
\end{split}
\end{equation}
where $M>0$ is the Mach number, $\gamma>0$ is a capillarity coefficient, $p=p(\rho)$ 
is a (normalised) non-monotone pressure function and ${\bf I}\in \R^{d\times d} $ is the  identity matrix.
We like to stress that this paper addresses  the low Mach number limit, i.e., $M \rightarrow 0,$ not the
sharp interface limit, i.e., $\gamma \rightarrow 0.$
Thus, we assume $\gamma$ to be small but fixed. We complement \eqref{eq:EK} with  initial data
\begin{equation}
 \rho(\cdot,0) = \bar \rho , \quad  \bv(\cdot,0) = \bar \bv 
\end{equation}
and boundary data
\begin{equation}\label{eq:boundary}
 \bv= 0, \quad \nabla \rho \cdot \bnu   =0  \quad \text{on } \del \Omega \times (0,T),
\end{equation}
 where $\bnu$ denotes the outward pointing unit normal vector to $\del \Omega$.
A consequence of these boundary conditions is the global conservation of mass
\[ \frac{d}{dt} \int_\Omega \rho(\cdot,t) \operatorname{d}{\bf x} =0.\] 

For sufficiently smooth initial data equation \eqref{eq:EK} has solutions with $\rho \in L^2((0,T),H^1(\Omega)),$ i.e., no shocks appear.
Essentially, this is due to the energy dissipation equality, see \eqref{eq:energy} below. For details concerning  well-posedness and regularity of solutions
we refer to \cite{BDD07}. 
In this work we restrict our attention to solutions of \eqref{eq:EK} which do not develop shocks.
We aim at constructing a scheme having the following properties:
\begin{itemize}
 \item Conservation of mass, see Remark \ref{rem:masscons}.
 \item Asymptotic preservation, i.e., it converges to the right limit for $M \rightarrow 0,$ see Lemma \ref{lem:fd:divergence} and \eqref{f:rhovnull}.
 \item Stability of the scheme in the low Mach limit, see Lemma \ref{lem:fd:stab}.
 \item Stability of the scheme for generic Mach numbers, see Lemma \ref{lem:fdns}.
 \item No timestep restriction involving $M$.
\end{itemize}

\begin{remark}[Momentum balance]
 While conservation of momentum would also be a desirable property of our scheme, it seems to be incompatible with the other properties we pursue, see \cite{GMP13}
for more details on this issue.
As there are no shocks it seems acceptable not to enforce conservation of momentum.
For a detailed exposition of the problems which may be caused by the use of nonconservative schemes we refer the reader to \cite{HL94}.
\end{remark}

\begin{remark}[Extension to the Navier-Stokes-Korteweg system]
 The scheme presented here will be easily extendable to the Navier-Stokes-Korteweg (NSK) system in case of small Reynolds numbers. 
The NSK system is obtained from  \eqref{eq:EK}  by
including a  viscosity term 
in the momentum balance.
This is elaborated upon in Remark \ref{rem:NSK}.
\end{remark}

Let us note that the phases (liquid/vapour)  can be identified with the density values for which 
$p'(\rho)>0$, cf., Figure \ref{fig:pressure}.
In addition, the pressure is related to the Helmholtz free energy density $W=W(\rho)$ via the Gibbs-Duhem equation
\begin{equation}\label{eq:GibbsDuhem}
 p(\rho) := \rho  W'(\rho) - W(\rho),\  \text{ in particular, } p'(\rho)=\rho W''(\rho).
\end{equation}
\begin{figure}[h]
\begin{center}
 \includegraphics[width=3cm]{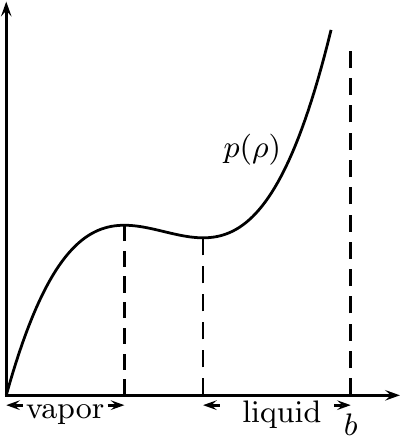},\hspace{2cm} \includegraphics[width=3cm]{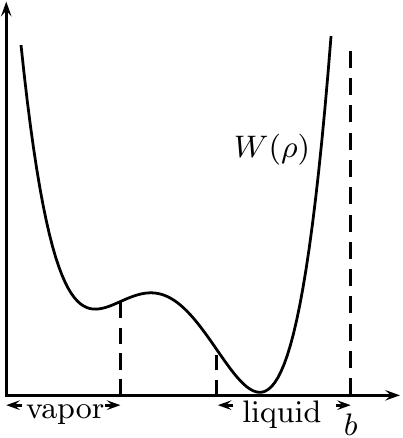}
\end{center}
\caption{
 Sketch of energy density and pressure.}
\label{fig:pressure}
\end{figure}
\noindent The non-monotone pressure and non-convex local part of the energy density are key features of the multi-phase character of the problem at hand.
They make the first order part of \eqref{eq:EK} hyperbolic-elliptic and make the well-posedness analysis as well as the construction of stable numerical schemes
rather involved. 
To be precise, we assume $W \in C^2 ((0,\infty),[0,\infty))$ and that there exist $0< \alpha < \beta< \infty$ such that
$W$ is strictly convex on $(0,\alpha) \cup (\beta,\infty)$ and strictly concave on $(\alpha,\beta).$
As
\begin{equation}\label{eq:Korteweg}
 \diver \left(\left(\rho \Delta \rho + \frac{1}{2} |\nabla \rho|^2\right) {\bf I} - \nabla \rho \otimes \nabla \rho \right) = \rho \nabla \Delta \rho
\end{equation}
it is straightforward to rewrite \eqref{eq:EK}, by introducing an auxiliary variable $\Lambda$, as
\begin{equation}\label{eq:EK2}
 \begin{split}
  \rho_t + \diver (\rho \bv) &=0\\
(\rho \bv)_t + \diver( \rho \bv \otimes \bv) + \frac{1}{M^2}\rho \nabla \Lambda &=0\\
  W'(\rho) - \gamma \Delta \rho - \Lambda &= 0 .
\end{split}
\end{equation}

Moreover, it is classical to check that energy is conserved for strong solutions of \eqref{eq:EK} equipped with boundary conditions \eqref{eq:boundary}, see 
\cite[e.g.]{GMP13}, i.e.,
\begin{equation}\label{eq:energy}
 \frac{d}{dt} \int_\Omega \frac{1}{M^2} \left( W(\rho) + \frac{\gamma}{2}|\nabla \rho|^2 \right) +\frac{1}{2} \rho |\bv|^2 \, \operatorname{d}{\bf x}  = 0.
\end{equation}
The non-local (gradient) term in the energy is responsible for including surface tension (capillary) effects in the model.
It also makes the interface smeared out and  thereby prevents the formation of shocks.
It follows from $\Gamma$-limit arguments that the thickness of the interfacial layer is proportional to $\sqrt{\gamma},$
see \cite[e.g.]{ORS90}.
The Euler-Lagrange equations  of the energy from \eqref{eq:energy} with a prescribed mass constraint are
\begin{equation}
 \lambda = \frac{1}{M^2} \left( W'(\rho) - \gamma \Delta \rho \right) , \quad  \bv= {\bf 0},
\end{equation}
where $\lambda \in \R$ is the Lagrange multiplier associated to the prescribed mass constraint.


The outline of the remainder of this paper is as follows: In
\S 2 we study  a formal  low Mach limit of the EK equations.
\S 3 is devoted to investigating a semi-discretisation in time.
This semi-discretisation is the basis of the fully discrete scheme which is stated and studied in \S 4.
To conclude, we present numerical experiments in \S 5.

\section{Low Mach Limit}\label{sec:lm}
While the low Mach limits of the Euler and Navier-Stokes equations have been rigorously studied in \cite[e.g.]{KM81,KM82,FN07}, less is known for the case of 
the hyperbolic-elliptic system with dispersion at hand. For a combined low Mach and sharp interface limit see \cite{HKK11}.
As the interest of this study is mainly numerical we 
consider a formal low Mach limit of \eqref{eq:EK2} in this section.
To this end, we assume expansions of all quantities in $M^2$
\begin{equation}\label{eq:expansions}
 \rho= \rho_0 + M^2 \rho_1 + o(M^2), \quad \Lambda=\Lambda_0 + M^2 \Lambda_1 + o(M^2) , \quad \bv = \bv_0 + o(1)
\end{equation}
where we assume $\rho_0,\rho_1,\bv_0,\Lambda_0,\Lambda_1$ to be sufficiently smooth for the subsequent calculations to make sense and $\rho_0>0$.
We assume the mass inside $\Omega$ to be prescribed independent of $M$ and thus
\[ \rho_1(\cdot,t) \in  H^1_m(\Omega):= \{\varphi \in H^1(\Omega) \, : \, \int_\Omega \varphi \, \operatorname{d}{\bf x} =0\}\]
for all $t \in (0,T).$ We impose the boundary conditions \eqref{eq:boundary} such that $\del_t \rho_0(\cdot,t) \in H^1_m(\Omega).$
By inserting \eqref{eq:expansions} into \eqref{eq:EK2} we immediately obtain
\begin{equation}
 \label{eq:consistency}
\Lambda_0 = W'(\rho_0) - \gamma \Delta \rho_0  = \text{const}.
\end{equation}
This is the leading order Euler-Lagrange equation.
In particular, \eqref{eq:consistency} holds for $t=0$ such that
the initial data need to be some extremum of the energy functional. We will only consider the more restrictive situation that 
for $\rho = \bar \rho_0$, i.e., the zeroth order of the initial data,  the bilinear form
\begin{equation}\label{bilinear}B_\rho :  H^1_m(\Omega) \times H^1_m(\Omega) \rightarrow \R, \quad  (\varphi,\psi) \mapsto \int_\Omega W''(\rho) \varphi \psi + \gamma \nabla \varphi \nabla \psi \, \operatorname{d}{\bf x}
\end{equation}
is coercive.

This condition, in particular, implies that $\bar \rho_0$ is an isolated minimiser of the leading order energy.
Computing the derivative of \eqref{eq:consistency} with respect to $t$  we find
\begin{equation}\label{eq:tderivative}
  W''(\rho_0) \del_t \rho_{0} - \gamma \del_t\Delta \rho_{0}= \text{const}.
\end{equation}
By continuity considerations we see that $B_\rho$ remains coercive for $\rho$ sufficiently near to $\bar \rho_0$ and 
$\rho_{0}(\cdot,t)$ is arbitrarily near to $\bar \rho_0$ for $t$ small enough.
Moreover, $\del_t \rho_0 \in H^1_m(\Omega)$, 
thus, \eqref{eq:tderivative} is uniquely solvable for small $t$ and the unique solution is $\del_t\rho_{0}=0.$
Via a continuation argument we get $\rho_0(\cdot,t)=\bar \rho_0$ for all $t \in [0,T).$

\noindent
Inserting $\del_t \rho_{0}= 0$ into \eqref{eq:EK2}$_1$ we infer that the leading order momentum is solenoidal, i.e.,
\begin{equation}\label{eq:divconstraint}
 \diver(\rho_0 \bv_0)=0.
\end{equation}
The low Mach limit is closed by the evolution equation for $\bv_0$ which reads
\begin{equation}\label{eq:vevolution}
 \rho_0 (\bv_0)_t + \rho_0 (\bv_0 \cdot \nabla ) \bv_0 + \rho_0 \nabla \Lambda_1 = 0,
\end{equation}
where $\Lambda_1$ can be determined via the  elliptic  (as $\rho_0>0$) equation
\begin{equation}
 - \diver (\rho_0 \nabla \Lambda_1) = \diver (\rho_0 (\bv_0 \cdot \nabla ) \bv_0)  
\end{equation}
such that it enforces the constraint \eqref{eq:divconstraint}.
Let us note that the role of the  chemical potential $\Lambda$ changed in the limit process.
 In the compressible case $\Lambda$ is given by a constitutive relation.
In the incompressible case $\Lambda$ is decomposed into a (fixed) background state $\Lambda_0$ given by the constitutive law and a Lagrange multiplier $\Lambda_1$.
To finish this section we state a stability result for the leading order velocity, which we will aim to recover as an inequality in the discrete setting.
\begin{lemma}[Conservation of kinetic energy]\label{lem:LMstab}
For a given $\rho_0\in H^1(\Omega)$ let $(\bv_0,\Lambda_1)$ be a strong solution of \eqref{eq:divconstraint}, \eqref{eq:vevolution} 
with $\bv_0|_{\del \Omega}= {\bf 0}$. Then,
\[\frac{d}{dt} \int_\Omega \rho_0 |\bv_0|^2 \, \operatorname{d}{\bf x} =0.\]
\end{lemma}

\begin{proof}
 Multiplying \eqref{eq:vevolution} by $\bv_0$ and integrating over $\Omega$ we obtain because of $\del_t \rho_{0}=0$
\begin{equation}\label{LM:cont:stab}
\frac{1}{2}\frac{d}{dt} \int_\Omega \rho_0   |\bv_0|^2\, \operatorname{d}{\bf x}  =
 - \int_\Omega \frac{1}{2} \rho_0 \bv_0 \cdot \nabla (|\bv_0|^2) + \rho_0 \bv_0 \cdot \nabla \Lambda_1 \, \operatorname{d}{\bf x}.
\end{equation}
 The assertion of the Lemma follows from \eqref{LM:cont:stab} using  integration by parts, \eqref{eq:divconstraint} and the boundary conditions.
\end{proof}
\begin{remark}[Expansion of the energy]
 Due to \eqref{eq:consistency}, \eqref{eq:boundary} and $\rho_1(\cdot,t) \in H^1_m(\Omega)$ we have
\begin{multline} \int_\Omega \frac{1}{M^2} ( W(\rho) + \frac{\gamma}{2} |\nabla \rho|^2) + \frac{1}{2} \rho |\bv|^2 \, \operatorname{d}{\bf x}\\
 = \int_\Omega \frac{1}{M^2} ( W(\rho_0) + \frac{\gamma}{2} |\nabla \rho_0|^2 )\, \operatorname{d}{\bf x} + 
\int_\Omega  \frac{1}{2} \rho_0 |\bv_0|^2 \, \operatorname{d}{\bf x} + \mathcal{O}(M^2),
\end{multline}
such that at least formally
\[ \frac{d}{dt} \int_\Omega \frac{1}{M^2} ( W(\rho) + \frac{\gamma}{2} |\nabla \rho|^2) + \frac{1}{2} \rho |\bv|^2 \, \operatorname{d}{\bf x}
 = \frac{d}{dt} \int_\Omega  \frac{1}{2} \rho_0 |\bv_0|^2 \, \operatorname{d}{\bf x} + \mathcal{O}(M^2),
\]
i.e., the leading order rate (of change) of the total energy is the rate of the kinetic energy.
\end{remark}


\section{A semi-discrete scheme}\label{sec:sd}
In this section we describe and investigate a semi-discretisation in time of \eqref{eq:EK} which can be used together with any space discretisation approach.
We show that this scheme converges to a stable discretisation of the incompressible problem determined in \S \ref{sec:lm} for fixed timestep sizes and $M \rightarrow 0$.
\subsection{Semi-discretisation in time}
The discretisation described here is inspired by the scheme for the compressible isothermal Euler equations in \cite{DT11}, where an elliptic equation for 
$\rho^\np$ and an explicit equation for $\bv^\np$ are derived.
Our generalisation of this approach leads to a Cahn-Hilliard like equation for $\rho;$ the discretisation of which might reintroduce an order 
$\mathcal{O}(M)$ 
timestep restriction, see \cite{BBG99}. 
To avoid such a constraint we decompose the double well potential $W$ as the difference of two convex $C^2$--functions $U,V$
\[ W(\rho)= U(\rho) - V(\rho).\]
We assume that $V'' (\rho) \geq \kappa_V >0$ for all $\rho >0$.
We subdivide the time interval $[0,T]$  into a partition of $N$
consecutive adjacent subintervals whose endpoints are denoted
$t_0=0<t_1<\ldots<t_{N}=T$. The $n$-th timestep is denoted $\tau_n=t_{n+1}-t_n$ and $\tau:=\max_{n=1,\dots,N} \tau_n.$   We will consistently use the shorthand
$F^n(\cdot):=F(\cdot,t_n)$ for a generic timedependent function $F$.
We propose the following time discretisation:
\begin{equation}\label{sch:EK1}
\begin{split}
 \rho^{n+1}-\rho^n + \tau_n \diver ((\rho\bv)^{n+1}) &=0 \\
(\rho \bv)^{n+1} - ( \rho \bv)^n + \tau_n \diver(\rho^n \bv^n \otimes \bv^n) + \frac{\tau_n}{M^2} \rho^n \nabla \Lambda^{n+1}
- \tau_n \mu_h \Delta \bv^n  &={\bf 0} \\
U'(\rho^{n+1})-V'(\rho^n) - \gamma \Delta \rho^{n+1}-  \Lambda^{n+1} &=0,
\end{split}
\end{equation}
where we require
\begin{equation}\label{sd:bcv}
 \nabla \rho^{n+1} \cdot \bnu =0, \quad \bv^{n+1}={\bf 0} \text{ on } \del \Omega
\end{equation}
and choose $\rho^0=\bar \rho,\, \bv^0=\bar \bv.$
Moreover, $\mu_h$ is an artificial viscosity coefficient. 
The idea of decomposing the energy into a part which is treated explicitly and a part which is treated implicitly can already be found in 
\cite{Eyr98,VLR03}.
Note that \eqref{sd:bcv} implies $\rho^\np - \rho^n \in H^1_m(\Omega).$

\noindent
In order to show that one timestep of \eqref{sch:EK1} can be decomposed into an implicit equation determining $\rho^\np,\, \Lambda^\np$ and an explicit expression for
$\bv^\np$ we insert the expression for 
$(\rho\bv)^{n+1}$ from \eqref{sch:EK1}$_2$ into \eqref{sch:EK1}$_1$ which yields
\begin{equation}\label{sch:EK2}
\begin{split}
 \rho^{n+1}-\rho^n + \tau_n \diver (  - \frac{\tau_n}{M^2} \rho^n \nabla \Lambda^{n+1}) &=
\Phi^n\\
U'(\rho^{n+1})-V'(\rho^n) - \gamma \Delta \rho^{n+1}-  \Lambda^{n+1} &=0,
\end{split}
\end{equation}
where
\[ \Phi^n := \tau_n \diver ( - ( \rho \bv)^n + \tau_n \diver(\rho^n \bv^n \otimes \bv^n) - \tau_n \mu_h \Delta \bv^n).\]
In this way we (implicitly) introduce higher (4th) order derivatives which require an additional boundary condition.
We introduce the following artificial boundary condition
\begin{equation}
 \label{bc:art}\nabla \Lambda^\np \cdot \bnu =0 \quad \text{on } \del \Omega 
\end{equation}
 which seems natural as \eqref{sch:EK2}  resembles one timestep in a semi-discretised Cahn-Hilliard equation with density dependent mobility. 
It is important to note that due to the explicit discretisation of the concave part of $W$ we get an elliptic system. 
An alternative would be to use a discretisation like in \cite{BBG99}. In that case we would need to choose the 
the parameter $\theta_c$ (introduced in \cite{BBG99}) carefully in order to avoid 
a timestep restriction of the form $\tau_n \lesssim M$, see \cite[Thrm. 2.1]{BBG99}.

Due to our discretisation of the double--well potential we have an elliptic problem for $(\rho^{n+1},\Lambda^{n+1})$, i.e., \eqref{sch:EK2},
 and $\bv^{n+1}$ is explicitly given by \eqref{sch:EK1}$_2$.
We will not investigate the well-posedness of \eqref{sch:EK2} here, but study it in the fully discrete case, see Lemma \ref{lem:wpscheme}.

\begin{remark}[Extension to NSK]\label{rem:NSK}
 The discretisation given in \eqref{sch:EK2} is easily extendable to the isothermal, compressible Navier-Stokes-Korteweg system, by substituting the
artificial viscosity $\mu_h$ by the physical viscosity or the reciprocal of the Reynolds number, in a non-dimensionalised setting.
Similarly $\Delta \bv^n$ might be replaced by $\diver ({\boldsymbol \sigma}_{NS}^n)$, where ${\boldsymbol \sigma}_{NS}$ denotes the full
Navier-Stokes stress tensor.
The explicit treatment of the viscous term is particularly adequate for high Reynolds numbers, see \cite{EL96}.
An implicit treatment of the viscosity is not possible in our framework as it would make the right hand side of \eqref{sch:EK2}$_1$ depend on $\bv^\np.$
\end{remark}

\subsection{The low Mach number limit}
Assuming the well--posedness of the scheme, we study its behaviour for $M \rightarrow 0$.
To this end, we assume the following expansions of the fields in $M^2$ for every $n \in \{0,\dots,N\}$
\begin{equation}\label{eq:num_expansion}
 \rho^n =\rho^n_0 + M^2 \rho_1^n + o(M^2), \quad  \Lambda^n =\Lambda^n_0 + M^2 \Lambda_1^n + o(M^2), \quad \bv^n = \bv_0^n + o(1),
\end{equation}
and compatibility of the initial data with the compatibility constraints, i.e.,
\begin{equation}\label{hsd}
 W'(\rho_0^0) - \gamma \Delta \rho_0^0 = \text{const},\quad
\diver(\rho_0^0 \bv_0^0)=0,\quad
\rho_0^0>0.
\tag{H$_{sd}$}
\end{equation}

\begin{lemma}[Semi-discrete AP property]\label{lem:sd:rhoconst}
Provided the solution of \eqref{sch:EK2}, \eqref{sch:EK1}$_2$ satisfies the expansion \eqref{eq:num_expansion} and the initial data fulfil \eqref{hsd}, then
\[ \rho_0^n =\rho_0^0 \quad  \text{ and } \quad  \diver(\rho_0^n \bv_0^n)=0  \text{ for all } n  \in \{0,\dots,N\}.\]
\end{lemma}
\begin{proof}
 The proof uses induction. For  $n=0$ 
the assertion becomes
\[ \rho_0^0 =\rho_0^0 \quad  \text{ and } \quad  \diver(\rho_0^0 \bv_0^0)=0\]
which is valid as we assume \eqref{hsd}.
For the induction step we have the induction hypothesis
\[ \rho_0^n =\rho_0^0 \quad  \text{ and } \quad  \diver(\rho_0^n \bv_0^n)=0.\]
In particular, this implies $\rho_0^n  >0.$
 Thus, the leading order of \eqref{sch:EK2}$_2$ and \eqref{bc:art} imply
\begin{equation}\label{eq:inductionstep1}
 U'(\rho_0^{n+1}) - \gamma \Delta \rho_0^{n+1} - V'(\rho_0^n) =\text{const}.
\end{equation}
By induction hypothesis and \eqref{hsd} we have
\begin{equation}\label{eq:inductionstep2}
  U'(\rho_0^{n}) - \gamma \Delta \rho_0^{n} - V'(\rho_0^n) =\text{const}.
\end{equation}
Computing the difference of \eqref{eq:inductionstep1} and \eqref{eq:inductionstep2} we obtain
\begin{equation}\label{eq:inductionstep3}
 U'(\rho_0^{n+1}) - U'(\rho_0^n) - \gamma \Delta (\rho_0^{n+1}-\rho_0^n) =\text{const}.
\end{equation}
Due to the convexity of $U$, the fact that $\rho_0^{n+1} - \rho_0^n \in H^1_m(\Omega)$ and Poincare's inequality we find  $\rho_0^{n+1} = \rho_0^n$
upon testing \eqref{eq:inductionstep3}
with $\rho_0^{n+1} - \rho_0^n.$
We note that the solution of the scheme still satisfies \eqref{sch:EK1}$_1$. Thus,
\[ \diver(\rho_0^\np \bv_0^\np) = \frac{\rho_0^{n}- \rho_0^\np }{\tau_n} =0.\]
\end{proof}

In view of Lemma \ref{lem:sd:rhoconst} our discretisation becomes a constrained evolution equation for $\bv_0$ with Lagrange multiplier $\Lambda_1:$
\begin{equation}\label{sch:EK3}
\begin{split}
 \diver(\rho_0^0 \bv_0^\np)&=0\\
\rho_0^0 (\bv_0^{n+1} - \bv_0^n) + \tau_n \diver (\rho_0^0 \bv_0^n \otimes \bv_0^n) +\tau_n \rho_0^0 \nabla \Lambda_1 &= \tau_n \mu_h \Delta \bv_0^n
\end{split}
\end{equation}
 which is a consistent discretisation of \eqref{eq:divconstraint}, \eqref{eq:vevolution}, i.e., the low Mach limit of the PDE system.

\subsection{Stability in the low Mach limit}
We will show stability of \eqref{sch:EK3}.

\begin{lemma}[Kinetic energy estimate]\label{lem:sd:stability}
 The solution $(\bv_0^n,\Lambda_1^n)_{n \in \{0,\dots,N\}}$ of \eqref{sch:EK3} satisfies
\begin{multline}\label{eq:sd:stability}
\int_\Omega \rho_0^0  |\bv_0^{n+1}|^2 \, \operatorname{d}{\bf x}= \int_\Omega \rho_0^0  |\bv_0^n|^2 - \tau_n (\rho_0^0 \bv_0^n) \cdot D\bv_0^n (\bv_0^{n+1} - \bv_0^n) \\
- 2 \mu_h \tau_n |D\bv_0^n|^2 - \tau_n \mu_h D\bv_0^n : (D\bv_0^{n+1} -D\bv_0^n) \, \operatorname{d}{\bf x}.
\end{multline}
\end{lemma}
\begin{remark}[Stability]
Let us stress two facts about the possible increase in energy allowed in  Lemma \ref{lem:sd:stability}:
\begin{enumerate}
\item We expect $\bv_0^{n+1} - \bv_0^n$  and $D\bv_0^{n+1} -D\bv_0^n$ to be of order $\tau,$ thus, the possible increase 
in energy per timestep is of order $\tau^2.$
Hence, the Lemma ensures the stability of the scheme in the low Mach limit independent of the actual value of $M.$
\item In the fully discrete setting we will have a discrete inverse inequality at our disposal which will enable us to prove that
our discrete version of
  $\int_\Omega \rho_0^0  |\bv_0^{n}|^2 \, \operatorname{d}{\bf x}$
is decreasing in $n,$ up to boundary terms.
\end{enumerate}
\end{remark}

 \begin{proof}[Proof of Lemma \ref{lem:sd:stability}] 
  We start by testing \eqref{sch:EK3}$_2$ by $(\bv_0^{n+1} + \bv_0^n).$ This gives using integration by parts and \eqref{sd:bcv}
\begin{multline}\label{eq:sd:stab1}
0= \int_\Omega \rho_0^0 ( |\bv_0^{n+1}|^2 - |\bv_0^n|^2) + \tau_n \diver(\rho_0^0 \bv_0^n \otimes \bv_0^n)\cdot \bv_0^n 
+ \tau_n \diver (\rho_0^0 \bv_0^n \otimes \bv_0^n)\cdot \bv_0^{n+1}\\
-\tau_n  \diver(\rho_0^0(\bv_0^{n+1}+\bv_0^n)) \Lambda_1^\np
- \mu_h \tau_n  (\Delta \bv_0^n) \cdot (\bv_0^n + \bv_0^{n+1})\, \operatorname{d}{\bf x}.
\end{multline}
Note that due to  Lemma \ref{lem:sd:rhoconst} 
\begin{equation}\label{eq:sd:stab2} \diver(\rho_0^0(\bv_0^{n+1}+\bv_0^n)) =0\end{equation}
and
\begin{multline}\label{eq:sd:stab3}
\int_\Omega \diver(\rho^0_0 \bv_0^n \otimes \bv_0^n )\cdot \bv_0^n \, \operatorname{d}{\bf x} = -\int_\Omega (\rho^0_0 \bv_0^n \otimes \bv_0^n ): D \bv_0^n \, \operatorname{d}{\bf x}\\
= - \frac{1}{2} \int_\Omega \rho^0_0 \bv_0^n \cdot \nabla (| \bv_0^n |^2)\, \operatorname{d}{\bf x} =  \frac{1}{2} \int_\Omega \diver(\rho^0_0 \bv_0^n) (| \bv_0^n |^2)\, \operatorname{d}{\bf x}=0.
\end{multline}
Using \eqref{eq:sd:stab2} and \eqref{eq:sd:stab3} in \eqref{eq:sd:stab1} we find
\begin{multline}\label{eq:sd:stab4}
0= \int_\Omega \rho_0^0 ( |\bv_0^{n+1}|^2 - |\bv_0^n|^2) + \tau_n (\rho_0^0 \bv_0^n) \cdot D\bv_0^n (\bv_0^{n+1} - \bv_0^n) \\
+ 2 \mu_h \tau_n |D\bv_0^n|^2 +\tau_n \mu_h D\bv_0^n : (D\bv_0^{n+1} -D\bv_0^n) \, \operatorname{d}{\bf x},
\end{multline}
concluding the proof.

\end{proof}
This finishes our considerations concerning the low Mach limit.
We have seen that the scheme converges to a stable approximation of the right set of equations, i.e., \eqref{eq:divconstraint} and \eqref{eq:vevolution}.

\subsection{Stability for generic Mach numbers}
Here we study the stability of the scheme in case of generic Mach numbers.
\begin{lemma}[Energy estimate]\label{lem:sdns}
Let \[\tau_n < \min\Big\{ \mu_h, \frac{\rho^n_{\min}\mu_h}{2 (\|\bv^n\|_\infty + \|\bv^\np\|_\infty )^2 \|\rho^\np\|_\infty^2}\Big\}\quad \text{with }
\rho^n_{\min} = \min_{x \in \Omega} \rho^n(x)\] 
where $\| \cdot \|_\infty:= \| \cdot \|_{L^\infty(\Omega)},$ then 
\begin{equation}
\begin{split}
&\int_\Omega \frac{1}{M^2} \left( W(\rho^{n+1})+\frac{\gamma}{2}|\nabla \rho^{n+1}|^2 \right) + \frac{1}{2} \rho^{n+1} |\bv^{n+1}|^2 \, \operatorname{d}{\bf x}\\
& \quad\leq \int_\Omega \frac{1}{M^2} \left( W(\rho^{n})+\frac{\gamma}{2}|\nabla \rho^{n}|^2 \right) + \frac{1}{2} \rho^{n} |\bv^{n}|^2 \, \operatorname{d}{\bf x}\\
& \quad+ \int_\Omega \frac{\tau_n^2}{2\kappa_VM^2} |\bv^\np|^2 |\nabla \Lambda^\np|^2
+ \tau_n^2|\diver(\rho^\np \bv^\np)|^2 |\bv^n|^4 
+ \frac{\mu_h \tau_n}{2} |D\bv^\np -D\bv^n|^2 \, \operatorname{d}{\bf x}.
\end{split}
\end{equation}
\end{lemma}
\begin{remark}[Stability] We like to stress that:
\begin{enumerate}
 \item The possible increase in energy per timestep is $\mathcal{O}(\tau^2).$
\item In the fully discrete case we will be able to control $\int_\Omega \frac{\mu_h \tau_n}{2} |D\bv^\np -D\bv^n|^2 \, \operatorname{d}{\bf x}$ via an inverse inequality.
\end{enumerate}
\end{remark}

\begin{proof}[Proof of Lemma \ref{lem:sdns}]
We multiply \eqref{sch:EK1}$_1$ with $\tfrac{1}{M^2}\Lambda^\np - \tfrac{1}{2} |\bv^\np|^2 $
 and \eqref{sch:EK1}$_2$ with $\bv^{n+1}$.
Integrating over $\Omega$ and summing both equations gives
\begin{equation}\label{eq:stab19}
\begin{split}
0=& \int_\Omega (\rho^\np - \rho^n) \left( \tfrac{1}{M^2}(U'(\rho^{n+1}) -\gamma \Delta \rho^{n+1} - V'(\rho^n)) - \tfrac{1}{2} |\bv^\np|^2 \right)\\
&\qquad+ \tau_n \diver(\rho^\np \bv^\np)  \left( \tfrac{1}{M^2}\Lambda^\np - \tfrac{1}{2} |\bv^\np|^2 \right)
+ \rho^\np |\bv^\np|^2 - \rho^n \bv^n \cdot \bv^\np \\ &\qquad+ \tau_n \diver(\rho^n \bv^n \otimes \bv^n)\cdot\bv^\np
 + \frac{\tau_n}{M^2} \rho^n \bv^{n+1} \cdot \nabla \Lambda^\np - \tau_n \mu_h \bv^\np\cdot \Delta \bv^n\, \operatorname{d}{\bf x}.
\end{split}
\end{equation}
Let us consider the terms in \eqref{eq:stab19} one by one. 
Since $U$ and $V$ are convex we have
\begin{equation}
\label{eq:stab20}
\begin{split}
U(\rho^\np) - U(\rho^n)
 &\leq (\rho^\np - \rho^n) U'(\rho^\np), \\
-V(\rho^\np) + V(\rho^n) &\leq -(\rho^\np - \rho^n) V'(\rho^n) - \frac{\kappa_V}{2} (\rho^\np-\rho^n)^2.
\end{split}
\end{equation}
Moreover, integration by parts, \eqref{sd:bcv} and Young's inequality imply
\begin{equation}
\label{eq:stab22}
- \int_\Omega (\rho^\np - \rho^n) \Delta \rho^\np \, \operatorname{d}{\bf x} = \int_\Omega |\nabla\rho^\np|^2 - \nabla \rho^n \cdot \nabla \rho^\np \, \operatorname{d}{\bf x}
\geq \frac{1}{2} \int_\Omega |\nabla\rho^\np|^2 - |\nabla\rho^n|^2 \, \operatorname{d}{\bf x}
\end{equation}
and, again by integration by parts and Young's inequality, we see
 \begin{multline}
 \label{eq:stab23}
 \Big| \frac{\tau_n}{M^2} \int_\Omega \diver(\rho^\np \bv^\np) \Lambda^\np\!
 + \rho^n \bv^\np\! \cdot\!\nabla \Lambda^\np  \, \operatorname{d}{\bf x}\Big|
 = \left|\frac{\tau_n}{M^2}\int_\Omega (\rho^n - \rho^\np) \bv^\np\! \cdot\!\nabla \Lambda^\np \, \operatorname{d}{\bf x} \right|\\
 \leq \frac{\kappa_V}{2M^2} \int_\Omega (\rho^\np -\rho^n)^2 \, \operatorname{d}{\bf x}  + \frac{\tau_n^2}{2\kappa_VM^2}  \int_\Omega |\bv^\np|^2
|\nabla \Lambda^\np|^2 \, \operatorname{d}{\bf x}. 
 \end{multline}
Using integration by parts we find
 \begin{equation}\label{eq:stab24}
\begin{split}
&\left|\tau_n \int_\Omega \diver( \rho^\np \bv^\np ) ( - \frac{1}{2} |\bv^\np|^2) + \diver(\rho^n \bv^n \otimes \bv^n)\cdot \bv^\np\, \operatorname{d}{\bf x}   \right|\\
& =\left| \tau_n \int_\Omega (\rho^\np \bv^\np \otimes \bv^\np): D\bv^\np - \rho^n (\bv^n \otimes \bv^n ):D\bv^\np \, \operatorname{d}{\bf x}   \right|\\
&\leq \tau_n \int_\Omega |\rho^\np -\rho^n| |\bv^n|^2 |D\bv^\np| \, \operatorname{d}{\bf x} + \tau_n \int_\Omega \rho^\np (|\bv^n|+ |\bv^\np|) |\bv^\np - \bv^n| |D\bv^\np|\, \operatorname{d}{\bf x} \\
& \leq \tau_n \int_\Omega \frac{\tau_n^2}{\mu_h} |\diver(\rho^\np \bv^\np)|^2 |\bv^n|^4 \, \operatorname{d}{\bf x} + \frac{\tau_n \mu_h}{4} \int_\Omega |D\bv^\np|^2 \, \operatorname{d}{\bf x}\\
& \quad +\frac{\tau_n}{\mu_h} \|\rho^\np\|_\infty^2 \int_\Omega | \bv^\np -\bv^n|^2 (|\bv^n|+ |\bv^\np|)^2 \, \operatorname{d}{\bf x} 
+ \frac{\tau_n \mu_h}{4} \int_\Omega |D\bv^\np|^2 \, \operatorname{d}{\bf x}.
\end{split}
 \end{equation}
In addition, by Young's inequality we have
\begin{multline}\label{eq:stab25}
\mu_h \tau_n \int_\Omega D\bv^\np : D\bv^n \, \operatorname{d}{\bf x} = \mu_h \tau_n \int_\Omega |D\bv^\np|^2 \, \operatorname{d}{\bf x} - \mu_h \tau_n \int_\Omega D\bv^\np : (D\bv^\np - D\bv^n) \, \operatorname{d}{\bf x}\\
\geq \frac{\mu_h \tau_n}{2} \int_\Omega |D\bv^\np|^2 \, \operatorname{d}{\bf x} - \frac{\mu_h \tau_n}{2} \int_\Omega |D\bv^\np - D\bv^n|^2\, \operatorname{d}{\bf x}.
\end{multline}
It also holds that
\begin{multline}\label{eq:stab26}
-\frac{1}{2} (\rho^\np - \rho^n) |\bv^\np|^2 + \rho^\np |\bv^\np|^2 -\rho^n \bv^n \cdot \bv^\np \\ 
= \frac{1}{2} \rho^\np |\bv^\np|^2 -\frac{1}{2} \rho^n |\bv^n|^2 
+ \frac{1}{2}\rho^n |\bv^\np - \bv^n|^2.
\end{multline}
Inserting \eqref{eq:stab20}--\eqref{eq:stab26} into \eqref{eq:stab19} we get
\begin{equation}\label{eq:stab27}
\begin{split}
0 &\geq \int_\Omega \frac{1}{M^2} \left(W(\rho^\np) - W(\rho^n) + \frac{\gamma}{2} |\nabla \rho^\np|^2 - \frac{\gamma}{2} |\nabla \rho^n|^2   \right)
 + \frac{1}{2} \rho^\np |\bv^\np|^2 
 -\frac{1}{2} \rho^n |\bv^n|^2 \\ &+ \frac{1}{2} \rho^n_{\min} |\bv^{n+1}- \bv^n|^2 
- \frac{\tau_n^2}{2\kappa_VM^2} |\bv^\np|^2 |\nabla \Lambda^\np|^2
- \frac{\tau_n^3}{\mu_h} |\diver(\rho^\np \bv^\np)|^2 |\bv^n|^4
 \\
&- \frac{\tau_n}{\mu_h} \|\rho^\np\|_\infty^2 (\|\bv^n\|_\infty + \|\bv^\np\|_\infty)^2 |\bv^\np -\bv^n|^2 
- \frac{\mu_h \tau_n}{2} |D\bv^\np -D\bv^n|^2 \, \operatorname{d}{\bf x}.
\end{split}
\end{equation}
The assertion of the Lemma follows from our assumptions on $\tau_n.$
\end{proof}


\section{The fully discrete scheme}\label{sec:fd}

In this section we consider a fully discrete finite difference scheme, which is based on the semi-discretisation investigated in \S \ref{sec:sd}.
We restrict ourselves to the case of a Cartesian mesh and perform all calculations in $2D$. The restriction to one space dimension as well as the extension to three
space dimensions is straightforward.
 For ease of presentation, we assume that our grid has the same meshsize $h$ in both space directions.
In particular, the computational domain will be $\Omega=[0,1]^2,$ and by ${\bf x}_\iij$ we denote $(ih,jh)^T \in [0,1]^2$ where
 we choose $h= \tfrac{1}{K}$ for some $K \in \mathbb{N}.$
For a generic field $f$ we denote our approximation of $f({\bf x}_\iij)$ by $f_\iij.$

\subsection{The fully discrete scheme}
To avoid too many indexes we decompose $\bv = (u,w)^T$ and introduce the following operators which we define
for some generic grid function $(f_\iij)_\iij$ or (grid) vector field $({\bf f}_\iij)_\iij$ with ${\bf f}_\iij = (f^1_\iij,f^2_\iij)^T$:
\begin{align}\label{def:operators1}
 (\nabla_h f)_\iij &:= \frac{1}{2h}( f_\ipj - f_\imj, f_\ijp  - f_\ijm)^T,\\
\label{def:operators2}(\widetilde \nabla_h f)_\iij &:= \frac{1}{h} (f_\ipj - f_\iij,  f_\ijp -  f_\iij)^T, \\
\label{def:operators3}(\diver_h {\bf f})_\iij &:=\frac{1}{2h}( f^1_\ipj - f^1_\imj + f^2_\ijp - f^2_\ijm),\\
\label{def:operators4}(D_h {\bf f})_\iij&:= \frac{1}{h} \left( \begin{array}{cc}
                                         f^1_\ipj - f^1_\iij & f^1_\ijp - f^1_\iij\\
                                         f^2_\ipj - f^2_\iij & f^2_\ijp - f^2_\iij
                                        \end{array}\right),\\
\label{def:operators5}(\Delta_h f)_\iij &:= \frac{1}{h^2} (f_\ipj + f_\imj + f_\ijp + f_\ijm - 4 f_\iij).
\end{align}
The domain of definition of the functions obtained in this way depends on the domain of definition of $f$ and ${\bf f}$,
e.g., if $f_\iij $ is defined for $(i,j)\in \{0,\dots,K\}^2$ then $(\nabla_h f)_\iij$ is defined for $(i,j)\in \{1,\dots,K-1\}^2.$
Let us note for later use that for any  ${\bf f}=({\bf f}_\iij)_{(i,j)\in \{0,\dots,K\}^2}$
 the discrete Jacobian $D_h$ fulfils the following inverse inequality
\begin{equation}\label{eq:invp}
 \sum_{i,j=0}^{K-1} |(D_h {\bf f})_\iij|^2 \leq \frac{8}{h^2} \sum_{i,j=0}^K |{\bf f}_\iij|^2.
\end{equation}
In addition, we use a rather specialised operator to discretise $\diver(\rho \bv \otimes \bv)$, i.e.,
\begin{multline}\label{def:divergence}
 (\widetilde \diver_h (\rho \bv \otimes \bv))_\iij\\ =
\frac{1}{4h} \Big((\bv_\iij + \bv_\ipj)( \rho_\iij u_\iij + \rho_\ipj u_\ipj)
                - (\bv_\iij + \bv_\imj)( \rho_\iij u_\iij + \rho_\imj u_\imj)\\
                + (\bv_\iij + \bv_\ijp)( \rho_\iij w_\iij + \rho_\ijp w_\ijp)
                - (\bv_\iij + \bv_\ijm)( \rho_\iij w_\iij + \rho_\ijm w_\ijm)
\Big).
\end{multline}
 We will study the following fully discrete scheme
\begin{equation}\label{f:rho}
\rho^\np_\iij - \rho^n_\iij + \tau_n \diver_h( \rho^\np \bv^\np)_\iij =0,
\end{equation}
\begin{equation}\label{f:rhov}
\rho^\np_\iij \bv^\np_\iij - \rho^n_\iij \bv^n_\iij 
+ \tau_n \widetilde \diver_h (\rho^n \bv^n \otimes \bv^n)_\iij
+ \rho^n_\iij \frac{\tau_n}{M^2} (\nabla_h \Lambda^\np)_\iij
 - \mu_h \tau_n (\Delta_h \bv^n)_\iij ={\bf 0},
\end{equation}
\begin{equation}\label{f:Lambdaalt}
	\Lambda^\np_\iij - U'(\rho_\iij^\np) + V'(\rho_\iij^n) + \gamma( \Delta_h \rho^\np)_\iij =0,
\end{equation}
\noindent
for $(i,j)\in\{0,\dots,K\}^2.$ 
We implement the boundary conditions for $\rho$ via a \emph{ghost cell} approach 
\begin{equation}\label{bc:rho}
\rho^\np_{-1,j}= \rho^\np_{0,j},\quad \rho^\np_{K+1,j}= \rho^\np_{K,j},  \quad \rho^\np_{i,-1}= \rho^\np_{i,0},\quad \rho^\np_{i,K+1}= \rho^\np_{i,K},
\end{equation}
for $(i,j)\in\{0,\dots,K\}^2$ and analogous for $\Lambda.$
We weakly enforce the boundary conditions on $\bv$ by setting
\begin{equation}\label{bc:v} 
\bv^\np_{-1,j}= -\bv^\np_{0,j},\quad \bv^\np_{K+1,j}=- \bv^\np_{K,j},\quad
 \bv^\np_{i,-1}=- \bv^\np_{i,0},\quad \bv^\np_{i,K+1}= -\bv^\np_{i,K},
\end{equation}
as in \cite{HJL12}.
Let us note that extending $\rho^n,\bv^n,\Lambda^n$ by these boundary conditions makes equations \eqref{f:rho} -- \eqref{f:Lambdaalt} well--defined.

\begin{remark}[Choice of discretisation]

\

\begin{enumerate}
 \item The time discretisation in \eqref{f:rho} - \eqref{f:Lambdaalt} is the same as in \S \ref{sec:sd}.
\item The advection term $\diver(\rho \bv \otimes \bv)$ is discretised such that the compatibility property in Lemma \ref{lem:Seite8}
holds.
\item The pressure gradient is discretised nonconservatively, see \cite{Die07,GMP13}.
\item The remaining spatial derivatives are discretised by central differences.
\end{enumerate}
\end{remark}

\begin{remark}[Conservation properties]\label{rem:masscons}
The scheme \eqref{f:rho} -- \eqref{bc:v} is mass conserving. In fact, it is an easy consequence of \eqref{f:rho},\eqref{bc:rho}
 and \eqref{bc:v} that
\begin{equation}\label{eq:masscons} \sum_{i,j=0}^K \rho^\np_\iij = \sum_{i,j=0}^K \rho^n_\iij  \quad \text{for all } n \in \{0,\dots,N-1\}.\end{equation}
We discretised the pressure gradient as $\rho^n \nabla \Lambda^\np$ which is nonconservative, thus, the scheme does not conserve momentum. 
Still, we like to stress that this is the only nonconservative term.
In particular, $\widetilde \diver (\rho \bv \otimes \bv)$ is conservative.
\end{remark}
\noindent
For later use, let us define the following sets
\begin{equation}
\begin{split}
 \del \Omega  &:= \big\{ (i,j) \in \{0,\dots,K\}^2 \, : \, i \in \{0,K\} \text{ or } j \in \{0,K\} \big\}, \\
 \bar \del \Omega  &:= \big\{ (i,j) \in \{-1,\dots,K+1\}^2 \, : \, i \in \{-1,K+1\} \text{ or } j \in \{-1,K+1\} \big\}.
\end{split}
\end{equation}

\subsection{Well--posedness of the scheme}
The well-posedness of the scheme results from its decomposition into an implicit equation for $\rho^\np$ and an explicit equation for $\bv^\np$.
To this end, we
insert the expression for $\rho_\iij^\np \bv_\iij^\np$  from \eqref{f:rhov} into \eqref{f:rho}
 and we obtain
\begin{equation}\label{f:ch}
\rho^\np_\iij - \frac{\tau_n^2}{M^2} \diver_h \Big( \rho^n (\nabla_h \Lambda^\np) \Big)_\iij = \rho^n_\iij -\tau_n \diver_h (\Phi^n)_\iij
\end{equation}
where $\Phi^n_\iij$ depends on quantities known at time $n$ only, i.e.,
\[ \Phi^n_\iij := 
\rho^n_\iij \bv^n_\iij - \tau_n \widetilde \diver_h (\rho^n \bv^n \otimes \bv^n)_\iij + \mu_h \tau_n (\Delta_h \bv^n)_\iij
\]
for $(i,j)\in \{0,\dots,K\}^2$ and $ \rho^n(\nabla_h \Lambda^\np)$ and $\Phi^n$ are extended to $(i,j) \in \bar \del \Omega$ using \eqref{bc:v}.

We will show that \eqref{f:ch}, \eqref{f:Lambdaalt} is uniquely solvable.
To do this, we  need  some definitions: By $\mathbb{V}$ we denote the space of all real valued tuples $(k_\iij)_{(i,j) \in \{0,\dots,K\}^2} $, i.e., 
$\mathbb{V}= \mathbb{R}^{(K+1)^2}$
and by $\mathbb{U} $ the subspace of $\mathbb{V}$ such that 
 $k_\iij$ is identical for all $\iij.$
The orthogonal complement of $\mathbb{U}$ in $\mathbb{V}$ with respect to the canonical scalar product is denoted by $\mathbb{V}_m.$
For any tuple $(q_\iij)_{(i,j) \in \{0,\dots,K\}^2}$ with $q_\iij >0$ for all $i,j$ the bilinear form
\begin{equation} B^h_q: (\mathbb{V}_m)^2 \rightarrow \R,\quad
 (l_\iij,k_\iij) \mapsto \sum_{i,j=0}^{K}  q_\iij (\nabla_h l)_\iij \cdot (\nabla_h k)_\iij
\end{equation}
is continuous and coercive, when $k,l$ are extended by \eqref{bc:rho}, in view of \eqref{def:operators1}.
Thus, it exists a linear, invertible operator
\begin{equation}\label{def:Gh}
 G^h_q : \mathbb{V}_m \rightarrow \mathbb{V}_m , \ \text{with} \quad 
- B^h_q(G_{q}^h(v),\chi)
 = \sum_{i,j=0}^K v_\iij \chi_\iij  \ \text{ for all } \chi \in \mathbb{V}_m.
\end{equation}
Because of the way $\Lambda^\np$ enters in \eqref{f:rhov} it is sufficient to extract the $\mathbb{V}_m$ part
 of $\Lambda^\np$ from \eqref{f:Lambdaalt}. 
Hence, we replace \eqref{f:Lambdaalt} by
\begin{equation}\label{f:Lambda}
	\Lambda^\np_\iij  + \mathbb{P}(- U'(\rho^\np) + V'(\rho^n) + \gamma( \Delta_h \rho^\np))_\iij =0,
\end{equation}
where $\mathbb{P}: \mathbb{V} \rightarrow \mathbb{V}_m$ is the orthogonal projection.
From \eqref{eq:masscons} we know $\rho^\np - \rho^n \in \mathbb{V}_m.$ 
Now we are in position to formulate our Lemma concerning existence and uniqueness of $\rho^\np,\Lambda^\np$.

\begin{lemma}[Well-posedness]\label{lem:wpscheme}
Let $\rho^n_\iij >0$ for all $(i,j)\in \{0,\dots,K\}^2$ then
there exists a  unique solution $(\rho^\np, \Lambda^\np)      \in (\rho^n + \mathbb{V}_m) \times \mathbb{V}_m$ of \eqref{f:ch}, \eqref{f:Lambda}.
\end{lemma}

\begin{remark}[Time step size]
 It is important to note that Lemma \ref{lem:wpscheme} does not require a timestep restriction. This is due to our splitting of $W.$
Apart from that the proof is similar to the well-posedness proof of the scheme in \cite{BBG99}.
\end{remark}

\begin{proof}[Proof of Lemma \ref{lem:wpscheme}]
As pointed out before $\rho^n(\nabla_h \Lambda^\np)$ (which replaces momentum in the mass conservation equation) is  extended to $\bar \del \Omega$
by \eqref{bc:v} and let $\psi \in \mathbb{V}_m$ be extended to $\bar \del \Omega$  by  \eqref{bc:rho}. Then, it holds
\begin{equation*}
\sum_{i,j=0}^K (\diver_h(\rho^n \nabla_h \Lambda^\np))_\iij\psi_\iij
= -\sum_{i,j=0}^K \rho^n_\iij(\nabla_h \Lambda^\np)_\iij (\nabla_h \psi)_\iij 
= \sum_{i,j=0}^{K} ((G^h_{\rho^n})^{-1}(\Lambda^\np))_\iij \psi_\iij.
\end{equation*}
Thus,  we  may equivalently pose  \eqref{f:ch} as 
\begin{equation}\label{eq:Gh}
\sum_{i,j=0}^K  \Big(\rho^\np_\iij  - \frac{\tau_n^2}{M^2} ((G^h_{\rho^n})^{-1}(\Lambda^\np))_\iij  - 
(\rho^n_\iij - \tau_n \diver_h( \Phi^n)_\iij)\Big) \psi_\iij =0
\quad \text{for all } \psi \in \mathbb{V}_m.
\end{equation}
As $\rho^\np - \frac{\tau_n^2}{M^2} ((G^h_{\rho^n})^{-1}\Lambda^\np)- \rho^n + \tau_n \diver_h( \Phi^n) \in \mathbb{V}_m$ 
is uniquely determined by \eqref{eq:Gh} we can apply 
$G^h_{\rho^n}$ to \eqref{eq:Gh}  and due to \eqref{f:Lambda}
\begin{equation}\label{eq:EL}
\frac{M^2}{\tau_n^2} (G^h_{\rho^n}(\rho^\np - \Psi^n))_\iij =\mathbb{P}( U'(\rho^\np) - V'(\rho^n)  
- \gamma ( \Delta_h \rho^\np))_\iij,
\end{equation}
where we used $\Psi^n:= \rho^n- \tau_n \diver_h(\Phi^n)$ for brevity.
When we extend functions in $\mathbb{V}_m$ by the boundary conditions \eqref{bc:rho}, equation
 \eqref{eq:EL} is the Euler--Lagrange equation to the following minimisation problem:
\begin{equation}\label{eq:min}
\min_{\psi \in \rho^n + \mathbb{V}_m} \sum_{i,j=0}^K\bigg[ U(\psi_\iij) + \frac{\gamma}{2} |(\widetilde \nabla_h \psi)_\iij|^2  -V'(\rho^n_\iij) \psi_\iij
 + \frac{M^2}{2\tau_n^2}  \rho^n_\iij |(\nabla_h (G^h_{\rho^n}(\psi - \Psi^n)))_\iij|^2 \bigg]
\end{equation}
where it is to be understood that $\psi$ and $G^h_{\rho^n}(\psi - \Psi^n)$ are extended by \eqref{bc:rho} to $\bar \del \Omega.$
This formulation as a minimisation problem is possible because for any $\vphi \in \mathbb{V}_m$ extended by \eqref{bc:rho}
\begin{multline}\label{eq:Ghtrick}
\frac{1}{2}\frac{d}{d\varepsilon}\sum_{i,j=0}^{K} \rho^n_\iij |(\nabla_h G^h_{\rho^n}(\rho^\np + \veps \vphi - \Psi^n))_\iij|^2\Big|_{\veps=0}\\
=\sum_{i,j=0}^{K} \rho^n_\iij \big( (\nabla_h G^h_{\rho^n}(\rho^\np- \Psi^n))_\iij \big) \cdot
\big( (\nabla_hG^h_{\rho^n}(\vphi))_\iij  \big)
=-\sum_{i,j=0}^{K} (G_{\rho^n}^h(\rho^\np - \Psi^n))_\iij \vphi_\iij.
\end{multline}
The existence of $\rho^\np,$ and thereby $\Lambda^\np,$ follows from the reformulation of \eqref{f:ch} as \eqref{eq:min}.
To show uniqueness let us assume there were two solutions $\rho^\np,\bar \rho^\np \in \rho^n + \mathbb{V}_m$ of \eqref{f:ch} and thereby of  \eqref{eq:EL}.
Due to \eqref{bc:rho} we do not get any boundary terms when performing summation by parts such that
\begin{multline}
0=\sum_{i,j=0}^K \bigg[\Big(U'(\rho^\np_\iij) - U'(\bar \rho^\np_\iij) - \frac{M^2}{\tau_n^2} (G^h_{\rho^n}(\rho^\np - \bar \rho^\np))_\iij\Big)\vphi_\iij \\
+ \gamma (\widetilde\nabla_h (\rho^\np - \bar \rho^\np))_\iij \cdot  (\widetilde\nabla_h \vphi)_\iij  \bigg]
\end{multline}
for all $\vphi \in \mathbb{V}_m$.
We choose $\vphi=\rho^\np-\bar \rho^\np$ and obtain
\begin{multline}\label{eq:uniq}
 0= \sum_{i,j=0}^K\bigg[(U'(\rho^\np_\iij) - U'(\bar \rho^\np_\iij))(\rho^\np_\iij - \bar \rho^\np_\iij) \\
        + \gamma |(\widetilde \nabla_h (\rho^\np -\bar\rho^\np))_\iij|^2  
+\frac{M^2}{\tau_n^2}\rho^n_\iij| (\nabla_h G^h_{\rho^n}(\rho^\np - \bar \rho^\np))_\iij|^2  \bigg]
\end{multline}
by using the second line of \eqref{eq:Ghtrick}. From \eqref{eq:uniq} we infer
\begin{equation}
0\geq \sum_{i,j=0}^K  |(\widetilde\nabla_h (\rho^\np - \bar \rho^\np))_\iij|^2 ,  
\end{equation}
which implies $\rho^\np - \bar\rho^\np \in \mathbb{V}_m \cap \mathbb{U}=\{0\}.$
\end{proof}

Once $\rho^\np,\Lambda^\np$ are determined from \eqref{f:ch}, \eqref{f:Lambda} equation \eqref{f:rhov} explicitly gives $\bv^\np$.
Thus, the scheme \eqref{f:rho} - \eqref{bc:v} is in fact well-posed.

\subsection{Asymptotic consistency of the scheme}
As in the semi-discrete case we commence our study of the properties of the scheme with the low Mach case.
To this end, we assume the following expansions
\begin{equation}
\rho^n_\iij = \rho_{0,i,j}^n + M^2 \rho^n_{1,i,j} + o(M^2) ,\  \bv_\iij^n=\bv^n_{0,i,j} + o(1) ,\
 \Lambda_\iij^n = \Lambda^n_{0,i,j} + M^2 \Lambda^n_{1,i,j} + o(M^2). 
\end{equation}
We also assume a fully discrete analogue of \eqref{hsd}, i.e.,
\begin{equation}\label{hfd}
 W'(\rho_{0,i,j}^0) -  ( \Delta_h \rho_0^0)_\iij  =\text{const},\quad
(\diver_h (\rho_0^0 \bv_0^0))_\iij=0,\quad
\rho_{0,i,j}^0>0,
\tag{H$_{fd}$}
\end{equation}
for all $(i,j)\in\{0,\dots,K\}^2.$

\begin{lemma}[AP property]\label{lem:fd:divergence}
Provided \eqref{hfd} holds, the solution of \eqref{f:rho} - \eqref{bc:v} satisfies
\[ \rho^n_0= \rho_0^0  \quad \text{ and } \quad \diver_h( \rho_0^n \bv_0^n)=0\quad
\text{ for all }n \in \{0,\dots,N\}.\]
\end{lemma}
\begin{proof}
The proof goes along the same lines as that of Lemma \ref{lem:sd:rhoconst}.
 It is based on induction and for $n=0$ the assertion coincides with \eqref{hfd}.
For the induction step we infer
from the leading order of \eqref{f:rhov} and \eqref{bc:rho} that $\Lambda^\np_0 \in \mathbb{U}$ such that because of \eqref{bc:rho}
\begin{equation*} 
\sum_{i,j=0}^K \Big( (U'(\rho^\np_{0,i,j}) - V'(\rho^n_{0,i,j}) )\psi_\iij + \gamma 
(\widetilde \nabla_h \rho_0^\np)_\iij \cdot(\widetilde\nabla_h \psi)_\iij\Big) =0
\end{equation*}
for all $\psi \in \mathbb{V}_m.$
By induction hypothesis this gives
\begin{equation*} 
\sum_{i,j=0}^K\Big( (U'(\rho^\np_{0,i,j}) - U'(\rho^n_{0,i,j}) )\psi_\iij + \gamma 
(\widetilde \nabla_h (\rho_0^\np-\rho_0^n))_\iij \cdot(\widetilde\nabla_h \psi)_\iij\Big) =0.
\end{equation*}
Choosing $\psi=\rho^\np_0 - \rho^n_0$ we obtain because of the convexity of $U$
\begin{equation}\label{0408} 
\sum_{i,j=0}^K  
|(\widetilde \nabla_h (\rho_0^\np - \rho_0^n))_\iij|^2 \leq0.
\end{equation}
Combining \eqref{0408} and \eqref{eq:masscons} we find $\rho^\np_0 - \rho^n_0 \in \mathbb{U} \cap \mathbb{V}_m = \{0\}.$
This proves the first claim of the lemma.
From the leading order of \eqref{f:rho} we obtain
\begin{equation}
\diver_h (\rho_0^\np \bv_0^\np) =0,
\end{equation}
which is the inductive step for the second assertion of the lemma.
\end{proof}

Thus, the scheme approximates the correct equations in the low Mach limit. In particular, as $\Lambda^\np_0 \in \mathbb{U}$, we have the following equation for $\bv_0^\np$
\begin{equation}\label{f:rhovnull}
\rho^0_\nij (\bv^\np_\nij - \bv^n_\nij )
+ \tau_n \widetilde \diver_h (\rho_0^0 \bv_0^n \otimes \bv_0^n)_\iij
+ \tau_n \rho^0_\nij  (\nabla_h \Lambda_1^\np)_\iij
 - \mu_h \tau_n (\Delta_h \bv^n_0)_\iij ={\bf 0}.
\end{equation}

\subsection{Stability in the low Mach limit}\label{sec:fd:lMstab}
Before we turn to the stability properties of the scheme let us consider the way we discretised $\diver(\rho \bv \otimes \bv)$ in 
\eqref{f:rhov}.
The operator $\widetilde \diver_h$ is deliberately constructed in such a way that the following lemma can be exploited.

\begin{lemma}[Compatibility]\label{lem:Seite8}
For any $(\rho_\iij)_{(\iij) \in \{0,\dots,K\}^2}, (\bv_\iij)_{(\iij) \in \{0,\dots,K\}^2}$ extended according to the boundary conditions \eqref{bc:rho}, \eqref{bc:v}
the following identity is satisfied
\begin{equation*}
\sum_{i,j=0}^K  ( \widetilde \diver_h (\rho \bv \otimes \bv))_\iij\cdot \bv_{i,j}
= \sum_{i,j=0}^K \frac{1}{2} |\bv_\iij|^2 ( \diver_h (\rho \bv))_\iij.
\end{equation*}
\end{lemma}

\begin{proof}
A straightforward calculation shows
\begin{equation}
 ( \widetilde \diver_h (\rho \bv \otimes \bv))_\iij \cdot \bv_{i,j}
=  \frac{1}{2}|\bv_\iij|^2 ( \diver_h (\rho \bv))_\iij
+ \frac{1}{4h}\big( c_\iij - c_\imj + \widetilde c_\iij - \widetilde c_\ijm  \big)
\end{equation}
with
\begin{equation*}
c_\iij = \bv_\ipj \cdot \bv_\iij ( \rho_\iij u_\iij + \rho_\ipj u_\ipj), \quad
\widetilde c_\iij = \bv_\ijp \cdot \bv_\iij (\rho_\iij w_\iij + \rho_\ijp w_\ijp).
\end{equation*}
Summing over $i,j=0,\dots,K$ completes the proof, because of \eqref{bc:rho} and \eqref{bc:v}.
\end{proof}

We are in position to prove the stability of the low Mach limit of the fully discrete scheme.

\begin{lemma}[Kinetic energy estimate]\label{lem:fd:stab}
Provided the timestep satisfies
\begin{equation}
 \label{eq:cfl}
\tau_n < \frac{h}{\|\bv_0^n\|_\infty}  \frac{ \rho_0^{\text{min}}}{ 9\|\rho_0^0\|^2_\infty + 8} \quad \text{ with }
 \rho_0^{\text{min}}:= \min_{(i,j) \in \{0,\dots,K\}^2} \rho_\nij^0,
\end{equation}
the  scheme  \eqref{f:rho} - \eqref{bc:v}  with $\mu_h := h \|\bv_0^n\|_\infty$ fulfils the following stability estimate
\begin{multline}\label{eq:fd:stab}
 \sum_{i,j=0}^{K} \rho_{0,i,j}^{0} |\bv_{0,i,j}^{n+1}|^2 \leq \sum_{i,j=0}^{K} \rho_{0,i,j}^0 |\bv_{0,i,j}^n|^2 \\
+ \frac{\mu_h \tau_n}{h^2} \sum_{(i,j)\in \del \Omega} | \bv_\nij^\np - \bv_\nij^n|^2 + 
 4\frac{\mu_h \tau_n}{h^2} \sum_{(i,j)\in \del \Omega} \bv_\nij^n(\bv_\nij^n - \bv_\nij^\np).
\end{multline}
\end{lemma}

\begin{remark}[Boundary conditions]
 The boundary terms in \eqref{eq:fd:stab} could be avoided if we enforced $\bv_\iij = {\bf 0}$ for $(i,j) \in \del \Omega \cup \bar \del \Omega.$
However, this would lead to $\rho_{0,0},\rho_{0,K},\rho_{K,0},\rho_{K,K}$ being (exactly) constant in time, which would be a very crude numerical artefact.
In any case the possible increase in energy is  expected to be of order $\mathcal{O}(\tau^2), $
as the boundary has codimension $1$ and $|\bv_\nij^n - \bv_\nij^\np|$ is expected to be of order $\mathcal{O}(\tau_n). $
\end{remark}

\begin{proof}[Proof of Lemma \ref{lem:fd:stab}]
Let us note the following consequences of the boundary conditions \eqref{bc:rho}, \eqref{bc:v} and  Lemmas \ref{lem:fd:divergence} and \ref{lem:Seite8}:
\begin{align}\label{eq:divtrick}
 \sum_{i,j=0}^K (\widetilde \diver_h (\rho_0^0 \bv_0^n \otimes \bv_0^n))_\iij \cdot \bv_\nij^n &=0\\
 \label{eq:Lambdatrick}\sum_{i,j=0}^K \rho_\nij^0 \bv_\nij^n \cdot (\nabla_h \Lambda_1^m)_\iij &=0
\end{align}
for all $m,n \in \{0,\dots,N\}.$
 We multiply \eqref{f:rhovnull} by $(\bv_\nij^\np+ \bv_\nij^n)$ and sum over $i,j=0,\dots,K$  such that we obtain using \eqref{eq:divtrick} 
and \eqref{eq:Lambdatrick}:
\begin{multline}\label{eq:fdstab1}
\sum_{i,j=0}^{K} \rho^0_\nij (|\bv^\np_\nij|^2 - |\bv^n_\nij|^2 )
+  \tau_n (\widetilde \diver_h (\rho_0^0 \bv_0^n \otimes \bv_0^n))_\iij \cdot (\bv_\nij^\np - \bv_\nij^n) \\
 - \mu_h\tau_n (\Delta_h \bv_0^n)_\iij \cdot (\bv^\np_\nij + \bv^n_\nij) =0.
\end{multline}
Due to the boundary conditions \eqref{bc:v} we have the following summation by parts results
\begin{equation}\label{pi:delta}
 \begin{split}
- \sum_{i,j=0}^{K} (\Delta_h \bv_0^n)_\iij \cdot  \bv^n_\nij &= \sum_{i,j=0}^{K-1}  |(D_h \bv_0^n)_\iij |^2 + \frac{2}{h^2} \sum_{(i,j)\in \del\Omega} |\bv^n_\nij|^2\\
- \sum_{i,j=0}^{K} (\Delta_h \bv_0^n)_\iij \cdot  \bv^\np_\nij &= \sum_{i,j=0}^{K-1}  (D_h \bv_0^n)_\iij:(D_h \bv_0^\np)_\iij 
 + \frac{2}{h^2} \sum_{(i,j)\in \del \Omega} \bv^n_\nij\cdot \bv^\np_\nij \\
 &\geq \sum_{i,j=0}^{K-1}  (D_h \bv_0^n)_\iij:(D_h \bv_0^\np)_\iij - \frac{1}{h^2} \sum_{(i,j)\in \del \Omega} |\bv^n_\nij- \bv^\np_\nij |^2.
 \end{split}
\end{equation}

Inserting \eqref{pi:delta} into \eqref{eq:fdstab1} we obtain
\begin{multline}\label{eq:fdstab2}
\sum_{i,j=0}^{K} \rho^0_\nij (|\bv^\np_\nij|^2 - |\bv^n_\nij|^2 )
+ \tau_n  (\widetilde \diver_h (\rho_0^0 \bv_0^n \otimes \bv_0^n))_\iij \cdot (\bv_\nij^\np - \bv_\nij^n) \\
 + \sum_{i,j=0}^{K-1} 2\mu_h\tau_n |(D_h \bv^n_0)_\iij|^2  - \mu_h \tau_n (D_h \bv^n_0)_\iij : (D_h (\bv^n_0 - \bv^\np_0))_\iij\\
 \leq
\frac{\mu_h \tau_n}{h^2} \sum_{(i,j)\in \del \Omega} |\bv^n_\nij- \bv^\np_\nij |^2 .
\end{multline}
To estimate the terms involving $\bv_0^\np - \bv_0^n$ we test \eqref{f:rhovnull} by $\bv_0^\np - \bv_0^n$. We find
using \eqref{eq:Lambdatrick} and \eqref{pi:delta}
 \begin{multline}
  \label{eq:fdstab3}
 \sum_{i,j=0}^{K} \rho^0_\nij | \bv_\nij^\np -\bv_\nij^n|^2 + \tau_n (\widetilde\diver_h (\rho_0^0 \bv_0^n \otimes \bv_0^n))_\iij \cdot (\bv^\np_\nij -\bv_\nij^n)\\
 +\sum_{i,j=0}^{K-1} \tau_n \mu_h (D_h \bv_0^n)_\iij : (D_h(\bv_0^\np-\bv_0^n))_\iij
\leq 2\frac{\mu_h \tau_n}{h^2} \sum_{(i,j)\in \del \Omega} \bv^n_\nij \cdot(\bv^n_\nij- \bv^\np_\nij ).
 \end{multline}
Due to Lemma \ref{lem:fd:divergence}  we have
\begin{multline}
4h (\widetilde \diver_h (\rho_0^0 \bv_0^n \otimes \bv_0^n))_\iij=
 4h (\widetilde \diver_h (\rho_0^0 \bv_0^n \otimes \bv_0^n))_\iij - 4h \bv_\nij^n (\diver_h(\rho_0^0\bv_0^n))_\iij=\\
(  \bv_\nipj^n - \bv_\nij^n)( \rho_\nij^0 u_\nij^n + \rho_\nipj^0 u_\nipj^n)
                - ( \bv_\nimj^n -\bv_\nij^n)( \rho_\nij^0 u_\nij^n + \rho_\nimj^0 u_\nimj^n)\\
                + ( \bv_\nijp^n-\bv_\nij^n)( \rho_\nij^0 w_\nij^n + \rho_\nijp^0 w_\nijp^n)
                - ( \bv_\nijm^n -\bv_\nij^n)( \rho_\nij^0 w_\nij^n + \rho_\nijm^0 w_\nijm^n)
\end{multline}
 for all $(i,j) \in \{0,\dots,K\}^2$ and a straightforward calculation  shows
\begin{equation}\label{eq:fdstab3b}
 \sum_{i,j=0}^{K} ((\widetilde \diver_h (\rho_0^0 \bv_0^n \otimes \bv_0^n))_\iij)^2 \leq
 \frac{9}{2} \|\rho_0^0\|_\infty^2 \|\bv_0^n\|_\infty^2 \sum_{i,j=0}^{K-1} |(D_h \bv_0^n)_\iij|^2.
\end{equation}
Employing   \eqref{eq:fdstab3b}  in  \eqref{eq:fdstab3} we obtain for arbitrary $\veps>0$
 \begin{multline}
  \label{eq:fdstab4}
 \sum_{i,j=0}^{K}\Big( \rho^0_\nij | \bv_\nij^\np -\bv_\nij^n|^2 - \tau_n \frac{\veps}{2} \|\rho_0^0\|_\infty | \bv_\nij^\np -\bv_\nij^n|^2\Big)\\
- \sum_{i,j=0}^{K-1}\tau_n \Big( \frac{9}{4\veps} \|\rho_0^0\|_\infty \|\bv_0^n\|_\infty^2 |(D_h \bv_0^n)_\iij|^2
+ \frac{ \mu_h}{2} |(D_h \bv_0^n)_\iij|^2 + \frac{ \mu_h }{2} |(D_h(\bv_0^\np - \bv_0^n))_\iij|^2 \Big)\\
 \leq 
2\frac{\mu_h \tau_n}{h^2} \sum_{(i,j)\in \del \Omega} \bv^n_\nij \cdot(\bv^n_\nij- \bv^\np_\nij ).
 \end{multline}
By the inverse inequality \eqref{eq:invp} we find
\begin{multline}\label{eq:fdstab5}
  \sum_{i,j=0}^{K} \rho_\nij^0 | \bv_\nij^\np -\bv_\nij^n|^2 \leq
\tau_n \left(  \frac{\veps}{2} \|\rho_0^0\|_\infty  + \frac{4\mu_h}{h^2}\right) \sum_{i,j=0}^{K} | \bv_\nij^\np -\bv_\nij^n|^2\\
+\tau_n \left( \frac{9}{4\veps} \|\rho_0^0\|_\infty \|\bv_0^n\|_\infty^2 + \frac{\mu_h}{2}\right)  
 \sum_{i,j=0}^{K-1} | (D_h \bv_0^n)_\iij|^2 +
 2\frac{\mu_h \tau_n}{h^2} \sum_{(i,j)\in \del \Omega} \bv^n_\nij \cdot(\bv^n_\nij- \bv^\np_\nij ).
\end{multline}
Let us choose $\veps =\tfrac{9}{h} \|\rho_0^0\|_\infty \|\bv_0^n\|_\infty.$
 Then, because of \eqref{eq:cfl} we have
\[   \tau_n\Big( \frac{\veps}{2} \|\rho_0^0\|_\infty  + \frac{4\mu_h}{h^2}\Big) \leq \frac{\rho_0^{\text{min}}}{2}\]
such that \eqref{eq:fdstab5} implies
\begin{multline}\label{eq:fdstab6}
\rho_0^{\text{min}} \sum_{i,j=0}^{K} | \bv_\nij^\np -\bv_\nij^n|^2 \\
\leq \tau_n \Big( \frac{h}{2}  \|\bv_0^n\|_\infty + \mu_h \Big)  
 \sum_{i,j=0}^{K-1} | (D_h \bv_0^n)_\iij|^2 
+ 4\frac{\mu_h \tau_n}{h^2} \sum_{(i,j)\in \del \Omega} \bv^n_\nij \cdot(\bv^n_\nij- \bv^\np_\nij ).
\end{multline}
Returning to  \eqref{eq:fdstab2} we find upon using \eqref{eq:fdstab3b} that for any $\bar \veps>0$
\begin{equation}\label{eq:fdstab7}
 \begin{split}
&   \sum_{i,j=0}^{K} \rho_\nij^0 |\bv_\nij^\np|^2 - \sum_{i,j=0}^{N} \rho_\nij^0 |\bv_\nij^n|^2\\
&\leq 
   \sum_{i,j=0}^{K} \Big(  \tau_n \frac{\bar \veps}{2}|\widetilde \diver_h (\rho_0^0 \bv_0^n \otimes \bv_0^n)_\iij|^2 + \frac{\tau_n}{2\bar \veps} |\bv_\nij^\np -\bv_\nij^n|^2\Big) 
+ \frac{\mu_h \tau_n}{h^2}\!\! \sum_{(i,j)\in \del \Omega} \!\!|\bv^n_\nij- \bv^\np_\nij |^2\\
&\quad  -\sum_{i,j=0}^{K-1} \Big(2\tau_n \mu_h |(D_h \bv_0^n)_\iij|^2 - \tau_n \mu_h |(D_h \bv_0^n)_\iij|^2 - \frac{\tau_n \mu_h}{4} | (D_h(\bv_0^\np - \bv_0^n))_\iij|^2\Big)\\
&=- \left( \tau_n \mu_h- \frac{9\tau_n \bar \veps}{4} \|\rho_0^0\|_\infty^2 \|\bv_0^n\|_\infty^2\right)\sum_{i,j=0}^{K-1}  |(D_h \bv_0^n)_\iij|^2\\
&\qquad+ \left( \frac{\tau_n}{2\bar \veps} + \frac{2\tau_n \mu_h}{h^2} \right) \sum_{i,j=0}^{K}  | \bv_\nij^\np - \bv_\nij^n|^2
+\frac{\mu_h \tau_n}{h^2} \sum_{(i,j)\in \del \Omega} |\bv^n_\nij- \bv^\np_\nij |^2.
 \end{split}
\end{equation}
Inserting \eqref{eq:fdstab6} into \eqref{eq:fdstab7} we get because of $\mu_h =\|\bv_0^n\|_\infty h$
\begin{multline}\label{eq:fdstab8}
  \sum_{i,j=0}^{K}\Big( \rho_\nij^0 |\bv_\nij^\np|^2 - \rho_\nij^0 |\bv_\nij^n|^2\Big)\\
\leq -\tau_n \left(  \mu_h- \frac{9 \bar \veps}{4} \|\rho_0^0\|_\infty^2 \|\bv_0^n\|_\infty^2
- \left( \frac{\tau_n}{2\bar\veps} + \frac{2\tau_n \mu_h}{h^2} \right) 
\frac{3h}{2\rho_0^{\text{min}}} \|\bv_0^n\|_\infty
\right)
\sum_{i,j=0}^{K-1}  |(D_h \bv_0^n)_\iij|^2\\
+ 
\frac{\mu_h \tau_n}{h^2} \sum_{(i,j)\in \del\Omega} |\bv^n_\nij- \bv^\np_\nij |^2
+ \Big(\frac{2\tau_n}{\bar \veps \rho_0^{\text{min}}}+ \frac{8\tau_n \mu_h}{h^2 \rho_0^{\text{min}}}   \Big)
\frac{\mu_h \tau_n}{h^2} \sum_{(i,j)\in \del \Omega} \bv^n_\nij \cdot(\bv^n_\nij- \bv^\np_\nij ).
 \end{multline}
Let us  choose $\bar \veps := \frac{2h}{9 \|\rho_0^0\|_\infty^2 
\|\bv_0^n\|_\infty}.$ Then, \eqref{eq:fdstab8} and \eqref{eq:cfl} imply
\begin{multline}\label{eq:fdstab9}
  \sum_{i,j=0}^{K}\Big( \rho_\nij^0 |\bv_\nij^\np|^2 - \rho_\nij^0 |\bv_\nij^n|^2\Big)\\
\leq -\tau_n\|\bv_0^n\|_\infty \frac{h}{2}\Big( 1
- 3 \Big(\frac{9\|\rho_0^0\|_\infty^2 }{4h} + \frac{2}{h} \Big) \frac{\tau_n}{\rho_0^{\text{min}}}
  \|\bv_0^n\|_\infty\Big)\sum_{i,j=0}^{K-1}  |(D_h \bv_0^n)_\iij|^2\\
 + \frac{\mu_h \tau_n}{h^2} \sum_{(i,j)\in \del \Omega} |\bv^n_\nij- \bv^\np_\nij |^2
+ 2\frac{\mu_h \tau_n}{h^2} \sum_{(i,j)\in \del \Omega} \bv^n_\nij \cdot(\bv^n_\nij- \bv^\np_\nij ).
 \end{multline}
Therefore, the assertion of the Lemma follows upon applying our assumption on $\tau_n$, i.e., \eqref{eq:cfl} again.
\end{proof}

\subsection{Stability for generic Mach numbers}
In this section we investigate the stability of the fully discrete scheme for generic Mach numbers.
As in the semi--discrete case the scheme does not necessarily diminish energy over time but
 the energy increase per timestep is  $\mathcal{O}(\tau^2).$
Therefore, for any given time interval the increase in energy over this interval goes to zero for $\tau$ going to zero.

\begin{lemma}[Fully discrete energy estimate]\label{lem:fdns}
For 
\[  \tau_n <\frac{h \|\bv^n\|_\infty  \rho^n_{\min}}{8(2\|\rho^\np\|_\infty^2(\|\bv^n\|_\infty + \|\bv^\np\|_\infty)^2 + \|\bv^n\|^2_\infty)} 
 \quad \text{ and } \quad \mu_h = \|\bv^n\|_\infty h,\] 
with $\rho^n_{\min}:= \min_{i,j=0,\dots,K} \rho^n_\iij, $ the solution $(\rho_\iij^{n+1},\bv_\iij^{n+1})$ of \eqref{f:rho}--\eqref{bc:v} satisfies
\begin{equation}\begin{split}
& \sum_{i,j=0}^K \Big(\frac{1}{M^2}\big(W(\rho_\iij^\np) + \frac{\gamma}{2}|(\widetilde \nabla_h \rho^\np)_\iij|^2\big) + \frac{1}{2} \rho_\iij^\np |\bv_\iij^\np|^2\Big)\\
- &  \sum_{i,j=0}^K \Big(\frac{1}{M^2}\big(W(\rho_\iij^n) + \frac{\gamma}{2}|(\widetilde \nabla_h \rho^n)_\iij|^2\big) + \frac{1}{2} \rho_\iij^n |\bv_\iij^n|^2\Big)\\
\leq 
& 8\tau_n^2\sum_{i,j=0}^K (\diver_h(\rho^\np \bv^\np)_\iij)^2 
+ \frac{\tau_n^2}{2\kappa_V M^2} \|\bv^\np\|_\infty^2 \sum_{i,j=0}^K | (\nabla_h \Lambda^\np)_\iij|^2\\
&+ \frac{ \tau_n}{h}\| \bv^n \|_\infty \sum_{(i,j)\in \del \Omega} |\bv^n_\iij- \bv^\np_\iij |^2.
\end{split}\end{equation}
\end{lemma}
\begin{remark}[Time step restriction]
 As we discretised the advection term in the momentum balance explicitly,  a timestep restriction proportional to
$\tfrac{h}{\|\bv\|_\infty}$ is to be expected.
Concerning the possible increase in energy we have seen in \S \ref{sec:fd:lMstab} that for small Mach numbers $\nabla_h \Lambda^\np \sim M^2$ such that the
$\tfrac{1}{M^2}|(\nabla_h \Lambda^\np)_\iij|^2$ term is well behaved provided the initial data satisfy the compatibility condition.

Let us stress that the timestep restriction is independent of the Mach number $M.$
\end{remark}

\smallskip

\begin{proof}[Proof of Lemma \ref{lem:fdns}]
 This proof has the same structure as the proof of Lemma \ref{lem:sdns}. We multiply \eqref{f:rho} by 
$ \tfrac{1}{M^2} \Lambda^\np_\iij - \tfrac{1}{2} |\bv_\iij^\np|^2$
and \eqref{f:rhov} by $\bv_\iij^\np.$ Adding both equations and summing $i,j=0,\dots,K$ we find
\begin{equation}\label{fd:ns1}
 \begin{split}
0= & \sum_{i,j=0}^K \Big( (\rho^\np_\iij - \rho^n_\iij)\big(\frac{1}{M^2} \big(U'(\rho^\np_\iij) - \gamma ( \Delta_h \rho^\np)_\iij - V'(\rho^n_\iij)\big)
 - \frac{1}{2} |\bv_\iij^\np|^2\big)\\
&+ \tau_n \diver_h( \rho^\np \bv^\np)_\iij (\frac{1}{M^2} \Lambda^\np_\iij -\frac{1}{2} |\bv_\iij^\np|^2)
+\rho^\np_\iij |\bv^\np_\iij|^2  - \rho^n_\iij \bv^n_\iij \cdot \bv^\np_\iij\\
&+ \tau_n \widetilde \diver_h (\rho^\np \bv^\np \otimes \bv^\np)_\iij \cdot \bv^\np_\iij
+ \rho^n_\iij \bv^\np_\iij \frac{\tau_n}{M^2} (\nabla_h \Lambda^\np)_\iij
 - \mu_h \tau_n (\Delta_h \bv^n)_\iij \cdot \bv^\np_\iij\Big).
\end{split}
\end{equation}
As  $U,\, V$ are convex we have for $(i,j) \in \{0,\dots, K\}^2$ 
\begin{equation}\label{fd:ns2}\begin{split}
&  (\rho^\np_\iij - \rho_\iij^n) U'(\rho_\iij^\np) \geq  U(\rho^\np_\iij) -U( \rho_\iij^n),\\
&-  (\rho^\np_\iij - \rho_\iij^n) V'(\rho_\iij^n) \geq -  \big(V(\rho^\np_\iij) -V( \rho_\iij^n) 
- \frac{\kappa_V}{2}(\rho^\np_\iij - \rho_\iij^n)^2\big),\\
& (\rho^\np_\iij - \rho_\iij^n)(-\frac{1}{2} |\bv_\iij^\np|^2) +\rho^\np_\iij |\bv^\np_\iij|^2  - \rho^n_\iij \bv^n_\iij \cdot \bv^\np_\iij\\
&\qquad \qquad\qquad \qquad\qquad \qquad= 
\frac{1}{2} \rho^\np_\iij |\bv^\np_\iij|^2 - \frac{1}{2}\rho^n_\iij |\bv^n_\iij|^2 + \frac{1}{2} \rho^n_\iij |\bv^\np _\iij - \bv_\iij^n|^2
\end{split}\end{equation}
and because of \eqref{bc:rho}
\begin{multline}\label{fd:ns2b}
- \gamma  \sum_{i,j=0}^K (\rho^\np_\iij - \rho_\iij^n) (\Delta_h \rho^\np)_\iij = 
\gamma \sum_{i,j=0}^K |(\widetilde \nabla_h\rho^\np)_\iij|^2 - (\widetilde\nabla_h\rho^n)_\iij \cdot (\widetilde\nabla_h\rho^\np)_\iij\\
\geq \frac{\gamma}{2} \sum_{i,j=0}^K |(\widetilde\nabla_h\rho^\np)_\iij|^2 -|(\widetilde\nabla_h\rho^n)_\iij|^2.
\end{multline}
In addition, we find using the boundary data \eqref{bc:rho}, \eqref{bc:v}
\begin{equation}\label{fd:ns3}\begin{split}
&\Big|\frac{\tau_n}{M^2}\sum_{i,j=0}^K  \diver_h( \rho^\np \bv^\np)_\iij  \Lambda^\np_\iij +  \rho^n_\iij \bv^\np_\iij \cdot  (\nabla_h \Lambda^\np)_\iij\Big|\\
& = \Big|\frac{\tau_n}{M^2}\sum_{i,j=0}^K  (\rho^n_\iij  - \rho^\np_\iij )\bv^\np_\iij \cdot  (\nabla_h \Lambda^\np)_\iij\Big|\\
& \leq \frac{\kappa_V}{2M^2} \sum_{i,j=0}^K  (\rho^n_\iij  - \rho^\np_\iij )^2 
+ \frac{\tau_n^2}{2\kappa_V M^2 } \|\bv^\np\|_\infty^2\sum_{i,j=0}^K |(\nabla_h \Lambda^\np)_\iij|^2.
\end{split}\end{equation}
To estimate the energy production by discretisation errors in the advection terms we use Lemma \ref{lem:Seite8} and obtain
\begin{equation}\label{fd:ns4}
 \begin{split}
   & \sum_{i,j=0}^K \diver_h( \rho^\np \bv^\np)_\iij (-\frac{1}{2} |\bv_\iij^\np|^2) +  \widetilde \diver_h (\rho^n \bv^n \otimes \bv^n)_\iij \cdot \bv^\np_\iij\\
  & = - \sum_{i,j=0}^K \big(\widetilde \diver_h (\rho^\np \bv^\np \otimes \bv^\np)_\iij - \widetilde \diver_h (\rho^n \bv^n \otimes \bv^n)_\iij \big)  \cdot \bv^\np_\iij\\
& =  \frac{-1}{4h}\sum_{i,j=0}^K ( {\bf a}^\np_\iij - {\bf a}^\np_\imj - {\bf a}^n_\iij + {\bf a}^n_\imj) \cdot \bv^\np_\iij 
                                + ( {\bf b}^\np_\iij - {\bf b}^\np_\ijm - {\bf b}^n_\iij + {\bf b}^n_\ijm) \cdot \bv^\np_\iij   
 \end{split}   
\end{equation}
where
\begin{align*}
 {\bf a}^n_\iij = (\bv_\iij^n + \bv_\ipj^n)(\rho^n_\iij u^n_\iij + \rho^n_\ipj u^n_\ipj), \quad
  {\bf b}^n_\iij = (\bv_\iij^n + \bv_\ijp^n)(\rho^n_\iij w^n_\iij + \rho^n_\ijp w^n_\ijp).
\end{align*}
Thus, we obtain using \eqref{bc:v}
\begin{equation}\label{fd:ns5}
 \begin{split}
   &\left| \sum_{i,j=0}^K \diver_h( \rho^\np \bv^\np)_\iij (-\frac{1}{2} |\bv_\iij^\np|^2) 
+  \widetilde \diver_h (\rho^n \bv^n \otimes \bv^n)_\iij \cdot \bv^\np_\iij\right|\\
& \leq \frac{1}{4h}  \sum_{i,j=0}^{K-1} | {\bf a}^\np_\iij - {\bf a}^n_\iij | \cdot |\bv^\np_\ipj -\bv_\iij^\np|
                                + | {\bf b}^\np_\iij - {\bf b}^n_\iij | \cdot |\bv^\np_\ijp -\bv_\iij^\np|\\
& = \frac{1}{8\mu_h} \sum_{i,j=0}^{K-1} \big(| {\bf a}^\np_\iij - {\bf a}^n_\iij |^2 + | {\bf b}^\np_\iij - {\bf b}^n_\iij |^2\big) +  \frac{\mu_h}{2}\sum_{i,j=0}^{K-1} |(D_h \bv^\np)_\iij|^2,
 \end{split}   
\end{equation}
A straightforward calculation gives
\begin{multline}\label{fd:ns6}
 \sum_{i,j=0}^{K-1} | {\bf a}^\np_\iij - {\bf a}^n_\iij |^2  \leq 
32  \|\rho^\np\|_\infty^2 (\| \bv^\np \|_\infty + \|\bv^n\|_\infty)^2  \sum_{i,j=0}^K | \bv^\np_\iij - \bv_\iij^n|^2\\
+ 32 \sum_{i,j=0}^K \|\bv^n\|_\infty^2 (\rho^\np_\iij -\rho^n_\iij)^2
\end{multline}
and an analogous estimate for $\sum_{i,j=0}^{K-1} | {\bf b}^\np_\iij - {\bf b}^n_\iij |^2.$
Inserting \eqref{fd:ns6} into \eqref{fd:ns5} we find
\begin{equation}\label{fd:ns7}
 \begin{split}
   &\left| \sum_{i,j=0}^K \diver_h( \rho^\np \bv^\np)_\iij (-\frac{1}{2} |\bv_\iij^\np|^2) 
+  \widetilde \diver_h (\rho^n \bv^n \otimes \bv^n)_\iij \cdot \bv^\np_\iij\right|\\
&\leq \frac{8}{\mu_h}  \|\rho^\np\|_\infty^2 (\| \bv^\np \|_\infty + \|\bv^n\|_\infty)^2  \sum_{i,j=0}^K | \bv^\np_\iij - \bv_\iij^n|^2\\
&+  \frac{8}{\mu_h}  \sum_{i,j=0}^K \|\bv^n\|_\infty^2 (\rho^\np_\iij -\rho^n_\iij)^2 +  \frac{\mu_h}{2} \sum_{i,j=0}^{K-1} |(D_h \bv^\np)_\iij|^2.
 \end{split}   
\end{equation}
Let us finally consider the artificial dissipation. We find using \eqref{pi:delta}
\begin{equation}\label{fd:ns8}
 \begin{split}
& - \sum_{i,j=0}^{K-1} (\Delta_h \bv^n)_\iij \cdot \bv^\np_\iij\\
   &\geq   \sum_{i,j=0}^{K-1} (D_h\bv^\np)_\iij : (D_h \bv^n)_\iij - \frac{1}{h^2} \sum_{(i,j)\in \del\Omega} |\bv^n_\iij- \bv^\np_\iij |^2\\
 &= \sum_{i,j=0}^{K-1}\Big( |(D_h\bv^\np)_\iij|^2 -   (D_h\bv^\np)_\iij : (D_h (\bv^\np-\bv^n))_\iij \Big)
- \frac{1}{h^2} \sum_{(i,j)\in \del\Omega} |\bv^n_\iij- \bv^\np_\iij |^2\\
&\geq \frac{1}{2} \sum_{i,j=0}^{K-1} |(D_h\bv^\np)_\iij|^2 - \frac{4}{h^2} \sum_{i,j=0}^K |\bv_\iij^\np-\bv_\iij^n|^2
-\frac{1}{h^2} \sum_{(i,j)\in \del\Omega} |\bv^n_\iij- \bv^\np_\iij |^2
 \end{split}   
\end{equation}
due to the inverse inequality \eqref{eq:invp}.
Inserting \eqref{fd:ns2} -- \eqref{fd:ns8} into \eqref{fd:ns1} we find
\begin{equation}\label{fd:ns9}\begin{split}
 0 \geq& \sum_{i,j=0}^K \frac{1}{M^2} \Big( W(\rho^\np_\iij) - W(\rho^n_\iij) +\frac{\gamma}{2} |(\widetilde\nabla_h \rho^\np)_\iij|^2
-\frac{\gamma}{2} |(\widetilde\nabla_h \rho^n)_\iij|^2 \Big)\\
& + \frac{1}{2} \rho^\np_\iij |\bv^\np_\iij|^2 - \frac{1}{2} \rho^n_\iij |\bv^n_\iij|^2\\
&- \frac{8\tau_n}{h}\|\bv^n\|_\infty\sum_{i,j=0}^K (\rho^\np_\iij - \rho^n_\iij)^2 
- \frac{\tau_n^2}{2\kappa_V M^2} \|\bv^\np\|_\infty^2 \sum_{i,j=0}^K | (\nabla_h \Lambda^\np)_\iij|^2\\
&+ \big( \frac{\rho^n_{\min}}{2} - 8\tau_n\frac{\|\rho^\np\|_\infty^2(\|\bv^n\|_\infty +\|\bv^\np\|_\infty)^2}{h \|\bv^n\|_\infty} 
- 4 \tau_n \frac{\|\bv^n\|_\infty}{h}   \big)
 \sum_{i,j=0}^{K}|\bv^\np_\iij - \bv^n_\iij|^2\\
&-\frac{\mu_h \tau_n}{h^2} \sum_{(i,j)\in \del \Omega} |\bv^n_\iij- \bv^\np_\iij |^2.
\end{split}\end{equation}
The assertion of the lemma follows from \eqref{fd:ns9} because of the  assumption on $\tau_n$ and \eqref{f:rho}.
\end{proof}

\section{Numerical experiments}
In this section we present  numerical experiments validating the desirable properties of the scheme described above.
In particular, we investigate the stability for generic and low Mach numbers and we compute the experimental order of convergence (EOC) 
for some examples.
The scheme was implemented in 1D and  2D using Matlab. The nonlinear systems were solved using the 'fsolve' command with default precision $10^{-7},$
if not stated otherwise.

\subsection{Stability for order $1$ Mach numbers}\label{sub:m1}
We consider the scheme \eqref{f:rho}-\eqref{bc:v} on the unit square $[0,1]^2$ and choose
\[ W(\rho)=(\rho-1)^2(\rho-2)^2 = (\rho^4 + 13\rho^2 + 4) - (6\rho^3 +12\rho)=:U(\rho)-V(\rho).\]
We consider the following initial datum
\[ \rho^0_\iij = \left\{\begin{array}{ccc}
                         3 &:& | \frac{i}{K} - \frac{1}{4}| + | \frac{j}{K} - \frac{1}{4}| \leq \frac{1}{4}\\
                         2 &:& | \frac{i}{K} - \frac{3}{4}| + | \frac{j}{K} - \frac{3}{4}| \leq \frac{1}{4}\\
                         1 &:& \text{else}
                        \end{array}
  \right.,  \qquad \bv_\iij^0= {\bf 0},\]
which is not near equilibrium.
The parameters are $M=1,$ $\gamma=9\cdot 10^{-4},$ $h=2.5\cdot10^{-2}$, $\tau=5\cdot10^{-4}$, i.e., we use a uniform timestep.
We show total energy and mass over time in Figure \ref{fig:m1}. The energy decreases (non-monotonically) and mass is conserved up to errors in the nonlinear solver. Snapshots of the solution are displayed in Figure \ref{fig:m1snap}.

 \begin{figure}[h]
 \includegraphics[width=.7\textwidth]{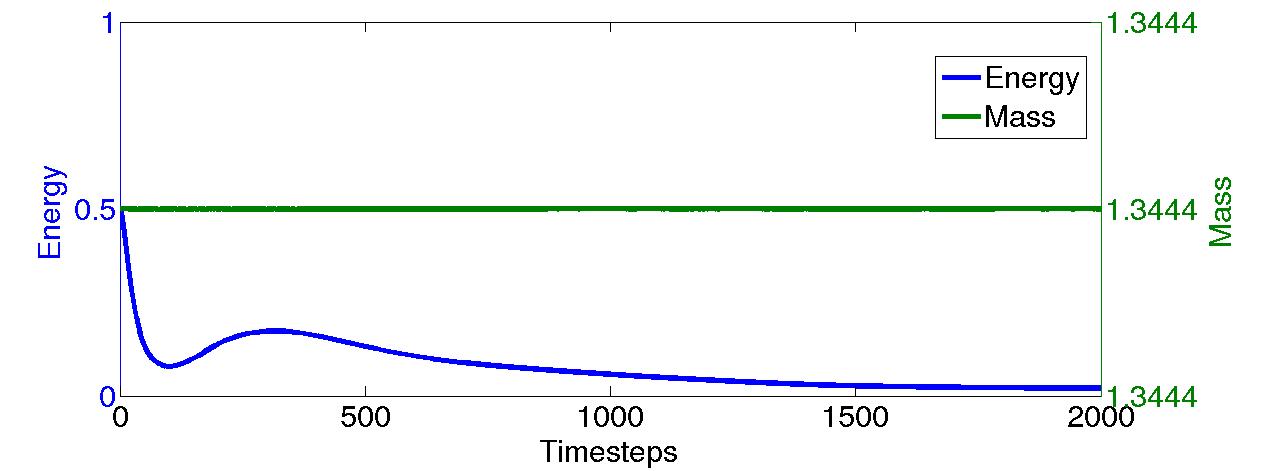}
  \caption{Energy and mass over time for the example in \S \ref{sub:m1}.}
\label{fig:m1}
 \end{figure}
 \begin{figure}[h]
\begin{center}
            \subfigure[][$t=0$]{
              \includegraphics[width=0.4\textwidth]
                              {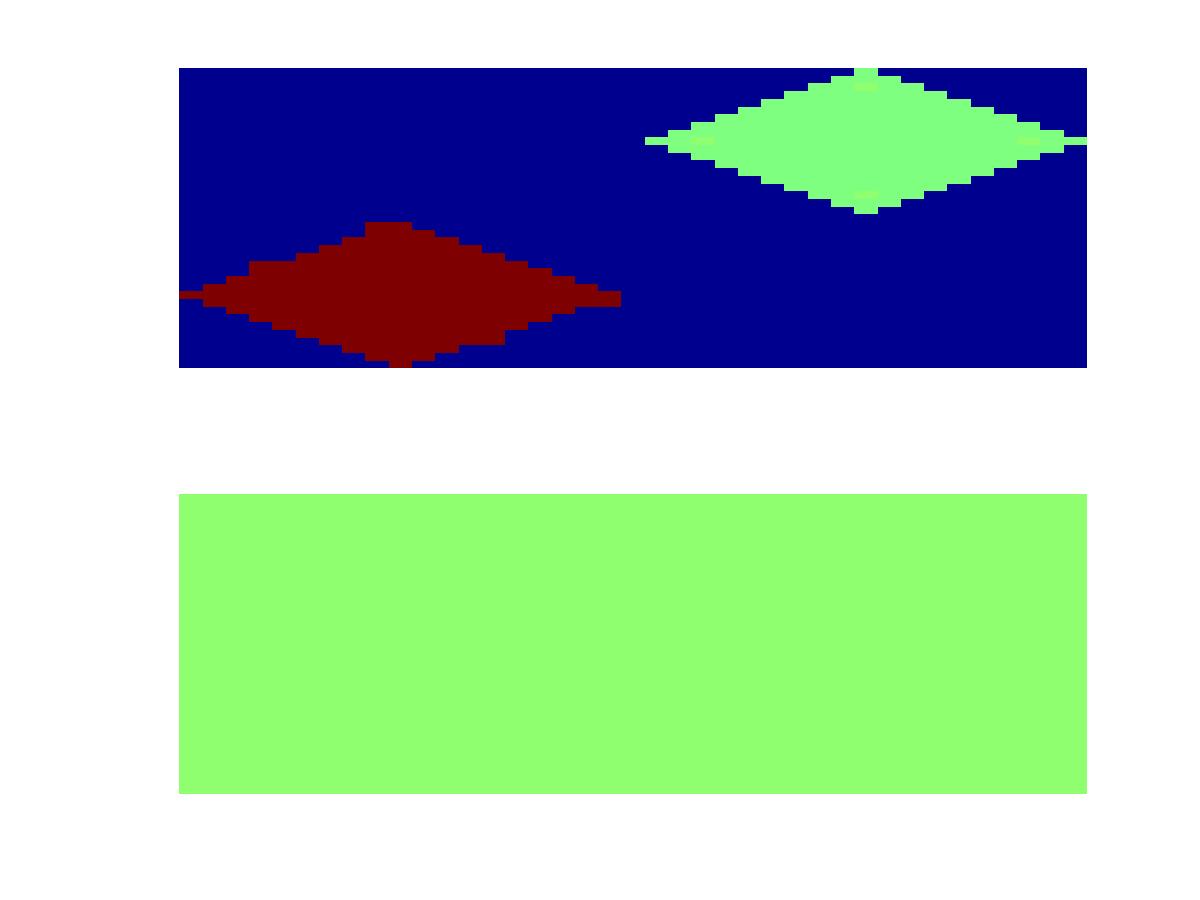}
            }
            \hfill
            \subfigure[][$t=0.05$]{
              \includegraphics[width=0.4\textwidth]
                              {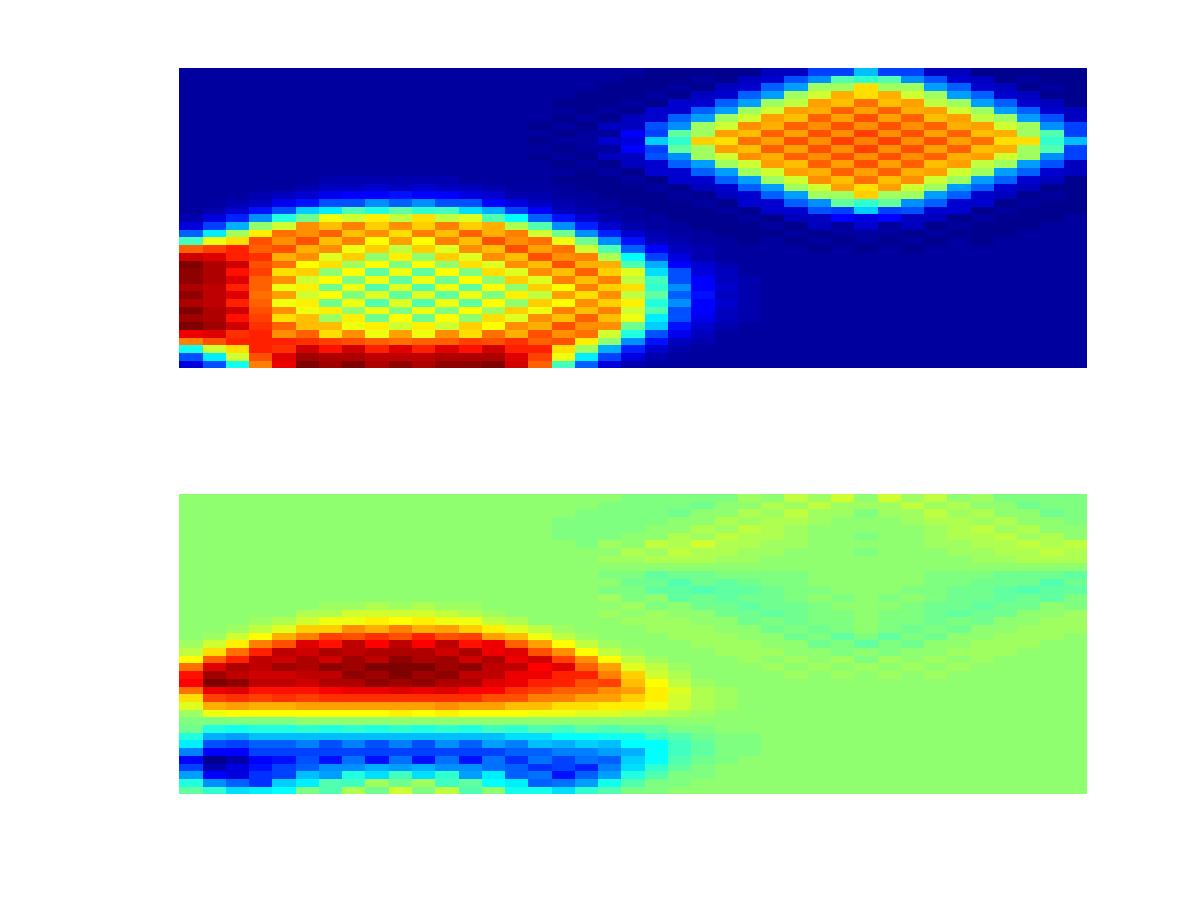}
}
            \subfigure[][$t=0.25$]{
              \includegraphics[width=0.4\textwidth]
                              {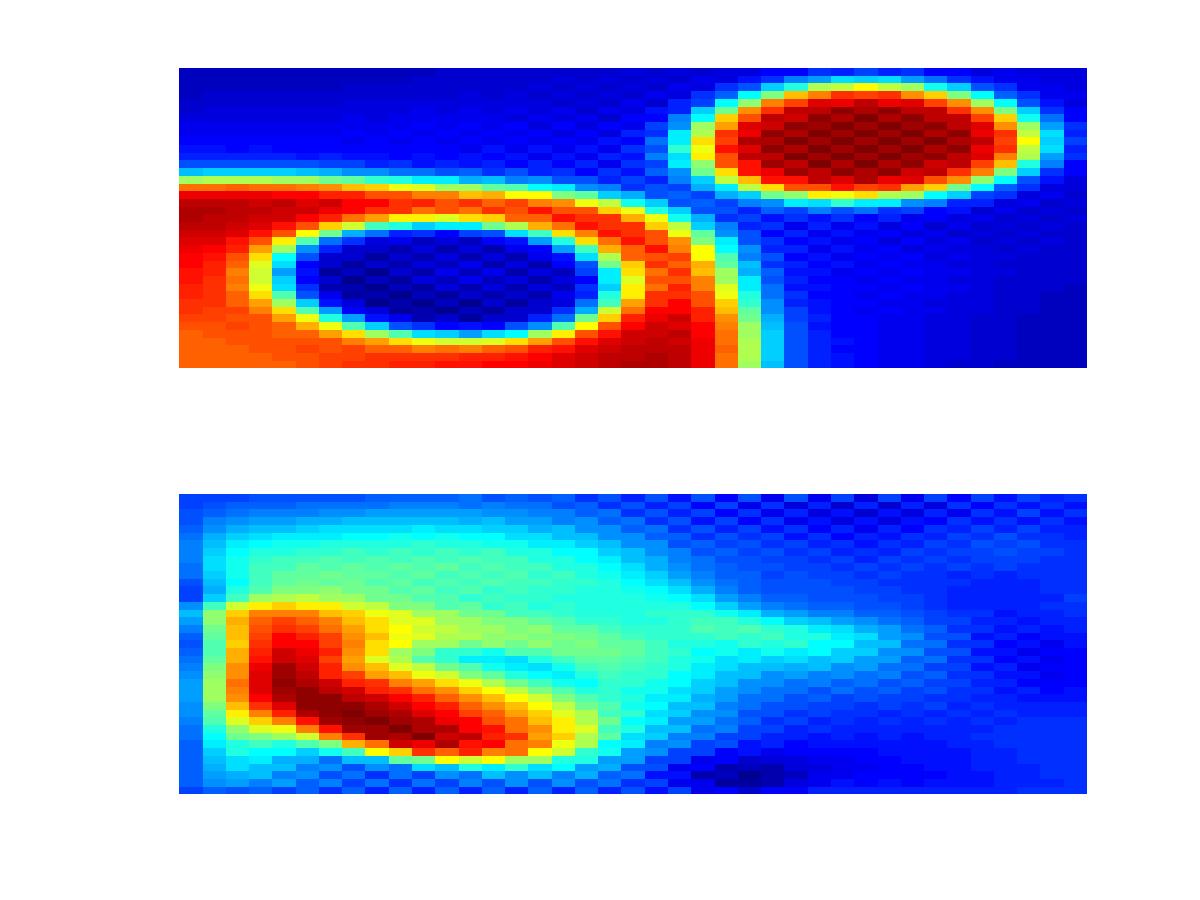}
            }
            \hfill
            \subfigure[][$t=1$]{
              \includegraphics[width=0.4\textwidth]
                              {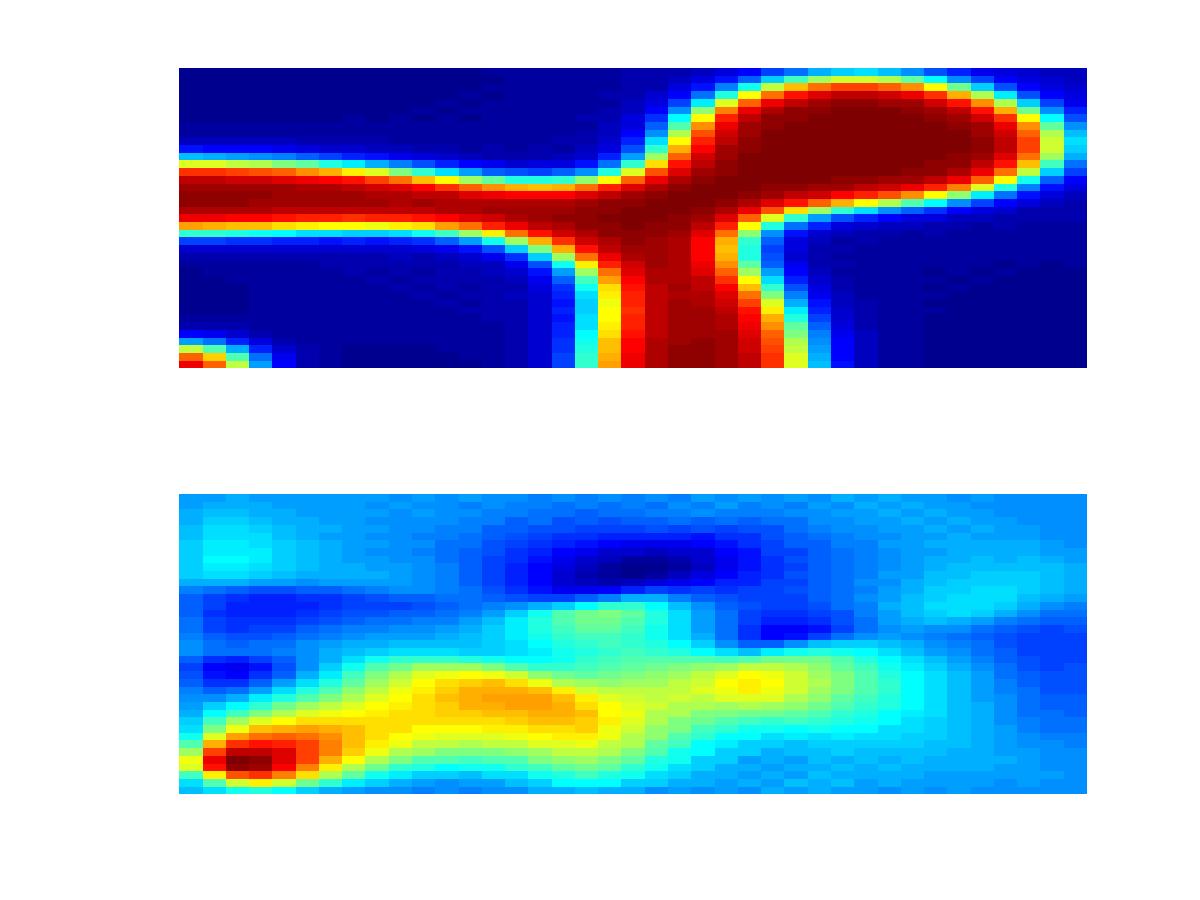}
}
          \end{center}
  \caption{Example described in \S \ref{sub:m1}. Snapshots of density (above) and horizontal velocity (below) at different times.
Red indicates high, green medium and blue low values.}
\label{fig:m1snap}
\end{figure}

\subsection{Stability for small Mach numbers}\label{sub:Mseq}
In this section we study a sequence of Mach numbers and initial densities approaching equilibrium for $M \rightarrow 0.$
We consider Mach numbers $M \in \{ 10^{-1}, 10^{-2}, 10^{-3}\}$ and initial data
\begin{equation}\begin{split}
 \rho_\iij^0 &= \frac{3}{2} - \big(\frac{1}{2} +4M\big)\cdot
                    \tanh\Big(\sqrt{\frac{2}{\gamma}}  \Big( \sqrt{\Big( \frac{i}{N} - \frac{1}{2} \Big)^2 + \Big( \frac{j}{N} - \frac{1}{2} \Big)^2 } - \frac{1}{4}  \Big)\Big)\\
   \bv_\iij^0 &= {\bf 0}
\end{split}\end{equation}
depending on the Mach number.
The other parameters are as in \S \ref{sub:m1}.
We  display the behaviour of (total) energy and kinetic energy over time in Figure \ref{fig:lm}.
The (total) energies are normalised by setting the energy at time zero to be one. This is done due to the fact that the initial energies differ and we are not interested
in absolute values of the energy but in its change in time.
The (total) energy is non-monotone in all three regimes. Still, its changes are rather small such that the schemes can be viewed as being stable.
Note that there is strong dissipation in case $M=.1$ while the energy increases above its initial value for the other two choices of $M.$ Initially the kinetic energy increases strongly (from zero) in all three regimes.
After the first at most 30 timesteps the kinetic energy decreases monotonically for all three choices of $M.$
In all three plots the lines for $M=10^{-2}$ and $M=10^{-3}$ are nearly identical.
In agreement with our analytic results we have better control of the kinetic energy for smaller Mach numbers.
 \begin{figure}[h]
 \includegraphics[width=.33\textwidth]{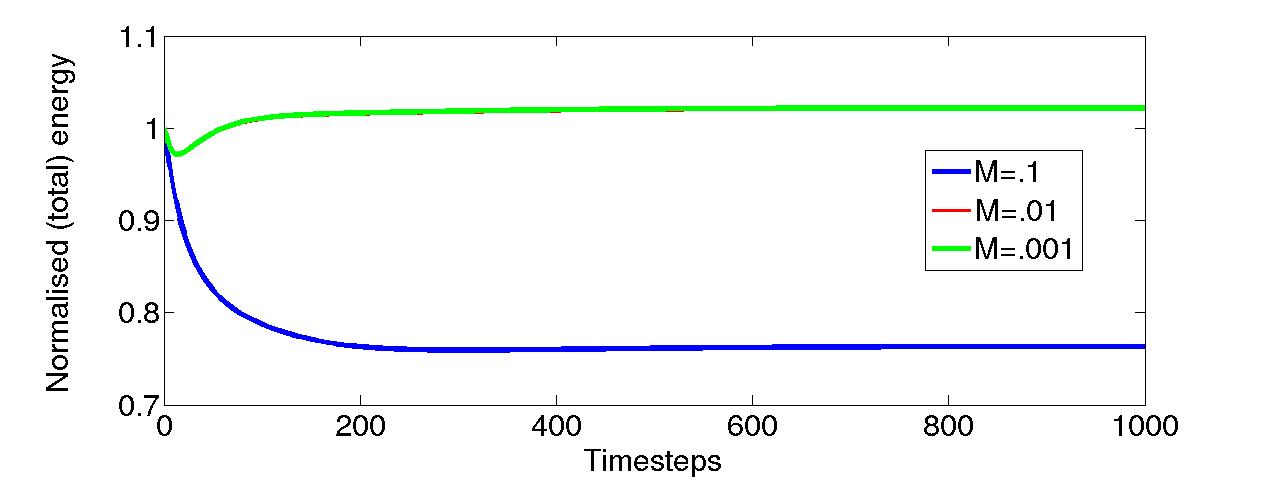}
  \includegraphics[width=.33\textwidth]{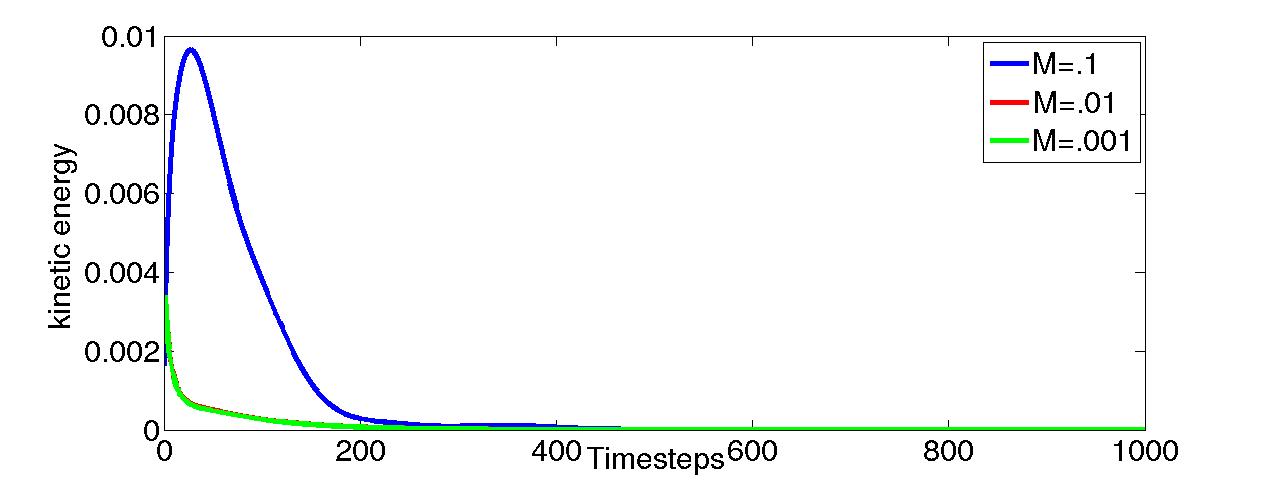}
 \includegraphics[width=.33\textwidth]{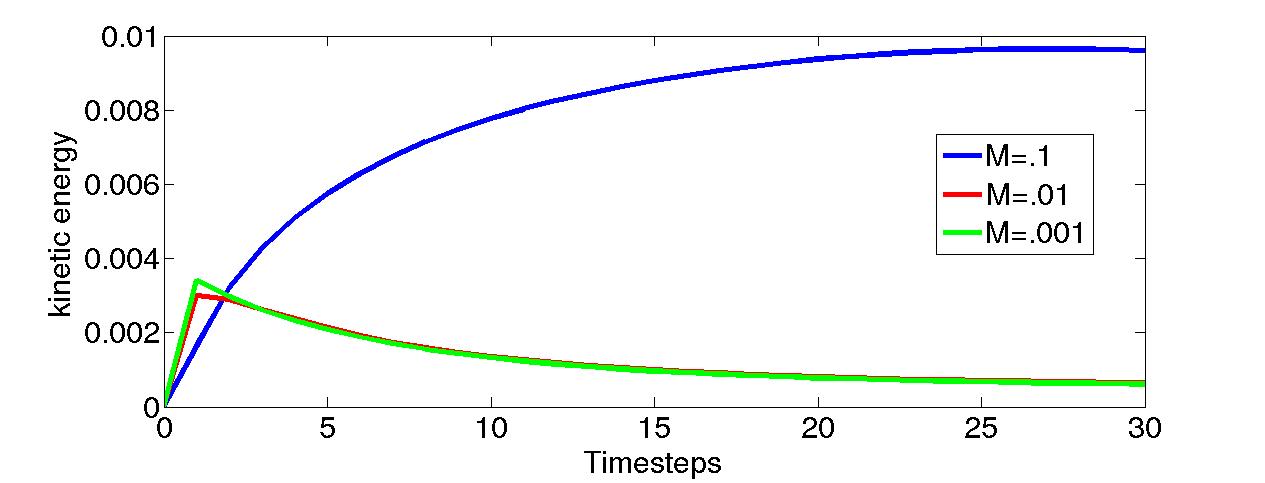}
  \caption{Left: Normalised total energy for different Mach numbers for the example described in \S \ref{sub:Mseq}. Middle: Kinetic energy for different Mach numbers for the example described in \S \ref{sub:Mseq}. Right: Kinetic energy for the first 30 timesteps for different Mach numbers for the example described in \S \ref{sub:Mseq}.}
\label{fig:lm}
 \end{figure}

\subsection{Convergence for order $1$ Mach numbers}\label{sub:cgm}
In this section we study the convergence properties of the scheme in 1D in a situation which is far away from equilibrium.
We consider the interval $[-1,1]$ as our computational domain and choose 
\begin{equation}
\bar \rho (x) = 1.5 + \tanh(\frac{2}{\sqrt{\gamma}}x),\qquad
\bar v(x)= 0
\end{equation}
with $\gamma=10^{-3}, M=1$ and $\tau= h/100.$

As can be seen from Figure \ref{dynM}  the dispersive nature of the problem and the fact that we are far away from equilibrium lead to
small oscillations near the interface, while the energy of the system decreases over time due to our discretisation.
This oscillatory behaviour of the solution leads to suboptimal convergence rates, see Table \ref{tab:dyn}.
There we show the relative $L^2$ errors of density $e_\rho^{\textrm{rel}}$ and velocity $e_v^{\textrm{rel}}$ at time $t=.0125$ for a given number of cells $K$ as well as the 
corresponding experimental order of convergence (EOC).
The errors are computed by comparison to a numerical solution on a mesh with $2560$ cells.

\begin{figure}[h]
\includegraphics[width=.29\textwidth]{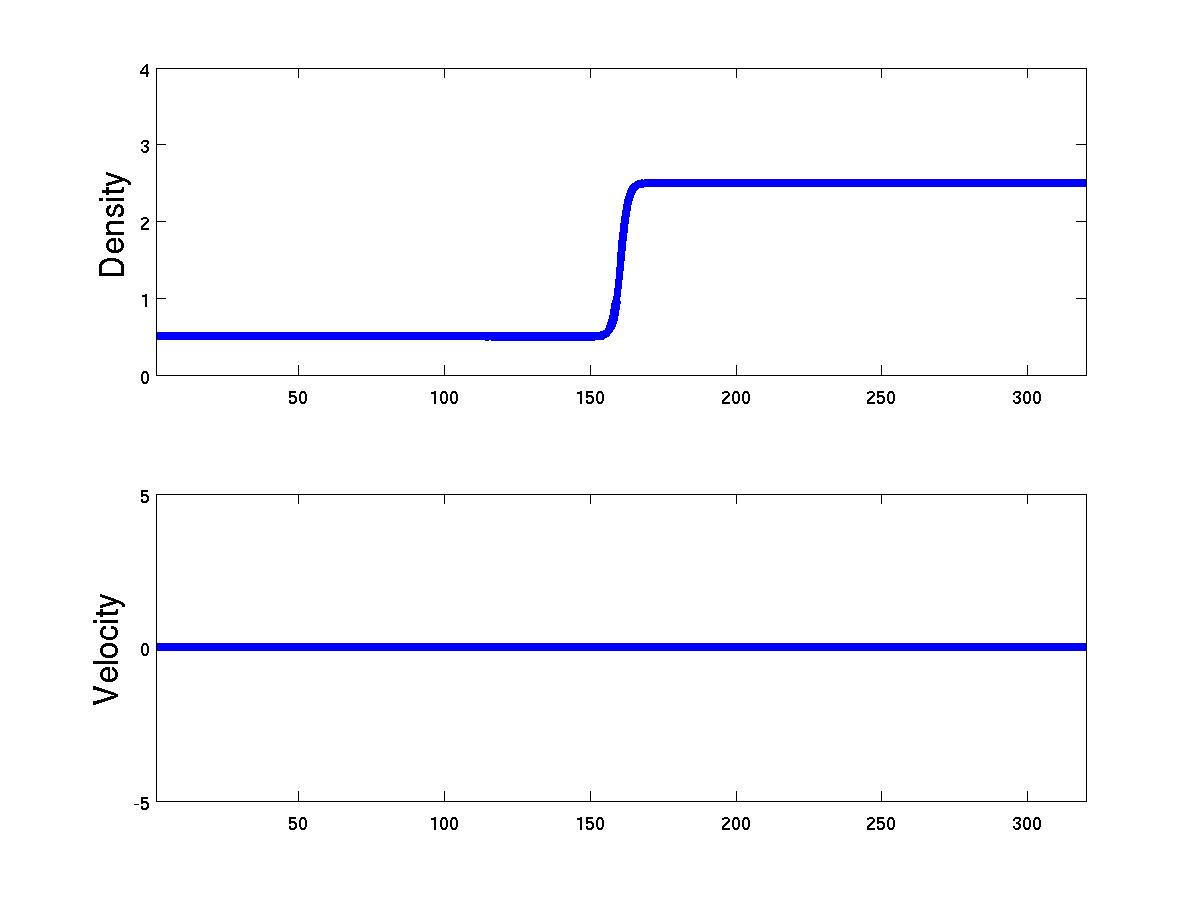}\includegraphics[width=.29\textwidth]{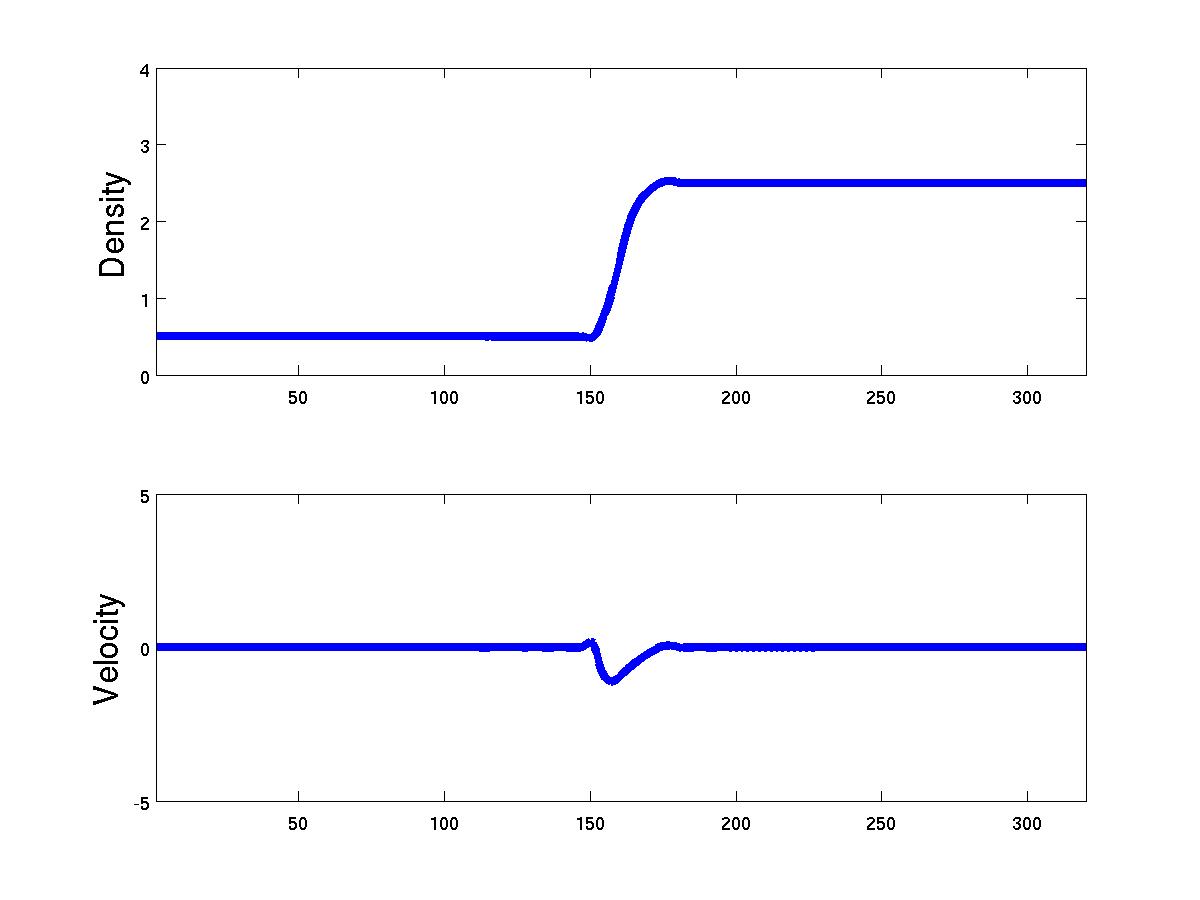}
\includegraphics[width=.4\textwidth]{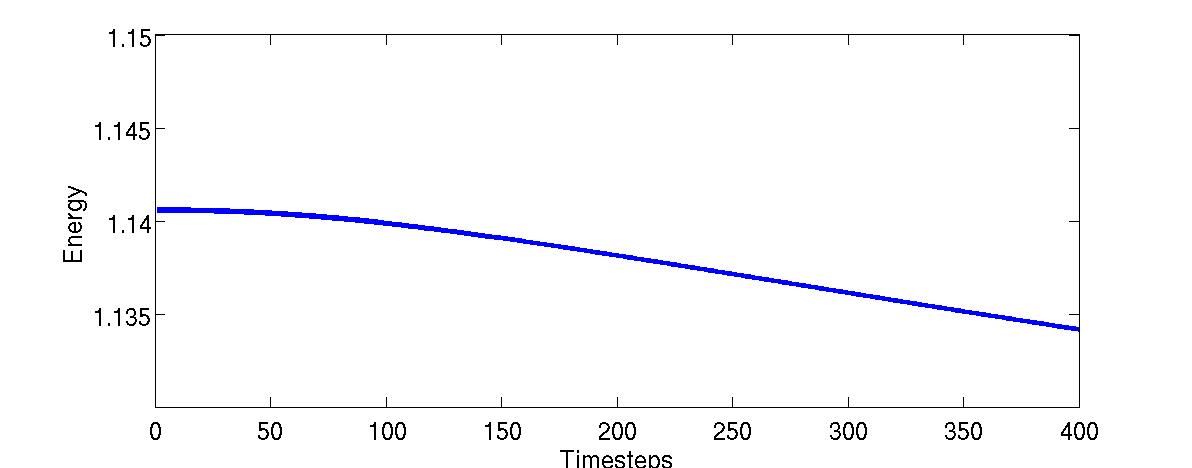}
\caption{
 Test from \S \ref{sub:cgm} with $K=320$: Left: Initial density and velocity. Middle: Density and velocity after 400 timesteps. Right: Energy over time.
}\label{dynM}
\end{figure}

\begin{table}[h]
\caption{Relative errors and EOCs for the test described in \S \ref{sub:cgm}. The dispersive structure of the problem leads to suboptimal convergence rates.}
\begin{tabular}{c|c|c|c|c}
 $K$ & $e_\rho^{\textrm{rel}}$ & EOC & $e_v^{\textrm{rel}}$ & EOC \\ \hline
 40 &    $5.545 \cdot 10^{-2}$ & -- & $3.167 \cdot 10^{-1}$ & --
 \\
 80 &    $2.806 \cdot 10^{-2}$ & $.98$ & $4.877 \cdot 10^{-1}$ & $.15$
 \\
160 &    $1.120 \cdot 10^{-2}$ & $1.3$ & $4.380 \cdot 10^{-1}$ & $.48$
 \\
320 &    $5.602 \cdot 10^{-3}$ & $1.0$ & $2.046\cdot 10^{-1}$ & $.62$
 \\
640 &    $2.713 \cdot 10^{-3}$ & $1.0$ & $1.096 \cdot 10^{-1}$ & $.90$
 \\
1280 &    $1.121 \cdot 10^{-3}$ & $1.3$ & $4.336 \cdot 10^{-2}$ & $1.3$
 \end{tabular}\label{tab:dyn} 
\end{table}
%

\subsection{Convergence for small Mach numbers}\label{sub:csm}
In this section we consider $M=.05$ and compare numerical solutions to a nearly  exact stationary solution which is given by
\begin{equation}\label{stationary}
\rho(x,t)= \frac{3}{2} + \frac{1}{2} \tanh\Big( \frac{x}{\sqrt{2\gamma}} \Big),\quad v(x,t)=0.
\end{equation}
It solves the PDE exactly and the error in the boundary conditions is negligible.
Initial conditions for the simulation are given by a pointwise evaluation of \eqref{stationary}.
We choose the timestep size as $\tau =\tfrac{h}{5}$  and the error tolerance of the nonlinear solver is set to $10^{-11}.$

The \emph{absolute} $L^2$ errors  at time $t=.25$ of density $e^{\textrm{abs}}_\rho$ and velocity $e^{\textrm{abs}}_v$ are displayed in Table \ref{tab:lm}. As the exact velocity is zero, it is not meaningful to consider relative errors here.
We observe that the density error converges very well, with a rather uniform convergence rate of $1$.
For the velocity error we observe good convergence except for two fine meshes where errors from the linear solver are amplified and lead to an increase in overall error
-- which is still rather small.
For larger error tolerance of the nonlinear solver the velocity error already starts increasing at larger meshwidth.

\begin{table}[h]
\caption{Absolute errors and EOCs for  the test described in \S \ref{sub:csm}. Round off errors in the nonlinear solver lead to negative convergence rates for
the velocity for fine meshes.}
\begin{tabular}{c|c|c|c|c}
 $K$ & $e^{\textrm{abs}}_\rho$ & EOC & $e_v^{\textrm{abs}}$ & EOC \\ \hline
 40 &    $4.314 \cdot 10^{-2}$ & -- & $3.403 \cdot 10^{-3}$ & --
 \\
 80 &    $1.997 \cdot 10^{-2}$ & $1.1$ & $3.850 \cdot 10^{-5}$ & $6.5$
 \\
160 &    $9.864 \cdot 10^{-3}$ & $1.0$ & $2.753 \cdot 10^{-6}$ & $3.8$
 \\
320 &    $4.891\cdot 10^{-3}$ & $1.0$ & $1.397 \cdot 10^{-6}$ & $1.0$
 \\
640 &    $2.385 \cdot 10^{-3}$ & $1.0$ & $2.567 \cdot 10^{-6}$ & $-0.89$
 \\
1280 &    $1.067 \cdot 10^{-3}$ & $1.2$ & $2.714 \cdot 10^{-6}$ & $-0.80$
\\
2560 &    $2.516 \cdot 10^{-5}$ & $5.4$ & $1.851 \cdot 10^{-6}$ & $0.55$
 \end{tabular}\label{tab:lm} 
\end{table}

\subsection{Linearised equation in timestep}\label{sub:lin}
In this last test we change the algorithm, in that we replace
$U'(\rho^\np) $ in the equation for $\rho^\np$ \eqref{f:rho} by $U'(\rho^n) + U''(\rho^n) (\rho^\np - \rho^n), $
i.e., we replace \eqref{f:Lambdaalt} by 
\begin{equation}\label{f:Lambdamod}
	\Lambda^\np_\iij - U'(\rho_\iij^n) - U''(\rho^n_\iij) (\rho^\np_\iij - \rho^n_\iij)+ V'(\rho_\iij^n) + \gamma( \Delta_h \rho^\np)_\iij =0.
\end{equation}
Thus, 
we only need to solve a linear problem in every timestep in order to determine $\rho^\np.$
While our analysis does not cover this modified algorithm, it leads to a considerable speedup in the computations.
We use the same data as in \S \ref{sub:cgm}.

The relative $L^2$ errors of density $e_\rho^{\textrm{rel}}$ and velocity $e_v^{\textrm{rel}}$ at time $t=.0125$ for a given number of cells $K$ as well as the 
corresponding experimental orders of convergence (EOC) are shown in Table \ref{tab:lin}.
The errors are computed by comparison to a numerical solution on a mesh with $2560$ cells.
Qualitatively the convergence properties look similar as but are less good than those in \S \ref{sub:cgm}.
This can be attributed to additional errors introduced by the linearisation.
In Figure \ref{fig:lin} we display snapshots of the solution after 200 and 400 timesteps and plot total energy over time, both for the case $K=320.$
We note that due to the linearisation of the convex part of the energy the energy of the numerical solutions is no longer decreasing. However,
the observed increase in energy is rather small.

\begin{table}[h]
\caption{Relative errors and EOCs for the test described in \S \ref{sub:lin}. The dispersive structure of the problems leads to suboptimal convergence rates.}
\begin{tabular}{c|c|c|c|c}
 $K$ & $e_\rho^{\textrm{rel}}$ & EOC & $e_v^{\textrm{rel}}$ & EOC \\ \hline
 40 &    $6.066 \cdot 10^{-2}$ & -- & $6.901 \cdot 10^{-1}$ & --
 \\
 80 &    $2.575 \cdot 10^{-2}$ & $1.2$ & $5.800 \cdot 10^{-1}$ & $.25$
 \\
160 &    $1.580 \cdot 10^{-2}$ & $.70$ & $4.820 \cdot 10^{-1}$ & $.27$
 \\
320 &    $1.056 \cdot 10^{-2}$ & $.58$ & $3.667 \cdot 10^{-1}$ & $.39$
 \\
640 &    $6.049 \cdot 10^{-3}$ & $.80$ & $2.296 \cdot 10^{-1}$ & $.68$
 \\
1280 &    $2.601 \cdot 10^{-3}$ & $1.2$ & $1.046 \cdot 10^{-1}$ & $1.1$
 \end{tabular}\label{tab:lin} 
\end{table}

\begin{figure}[h]
\includegraphics[width=.29\textwidth]{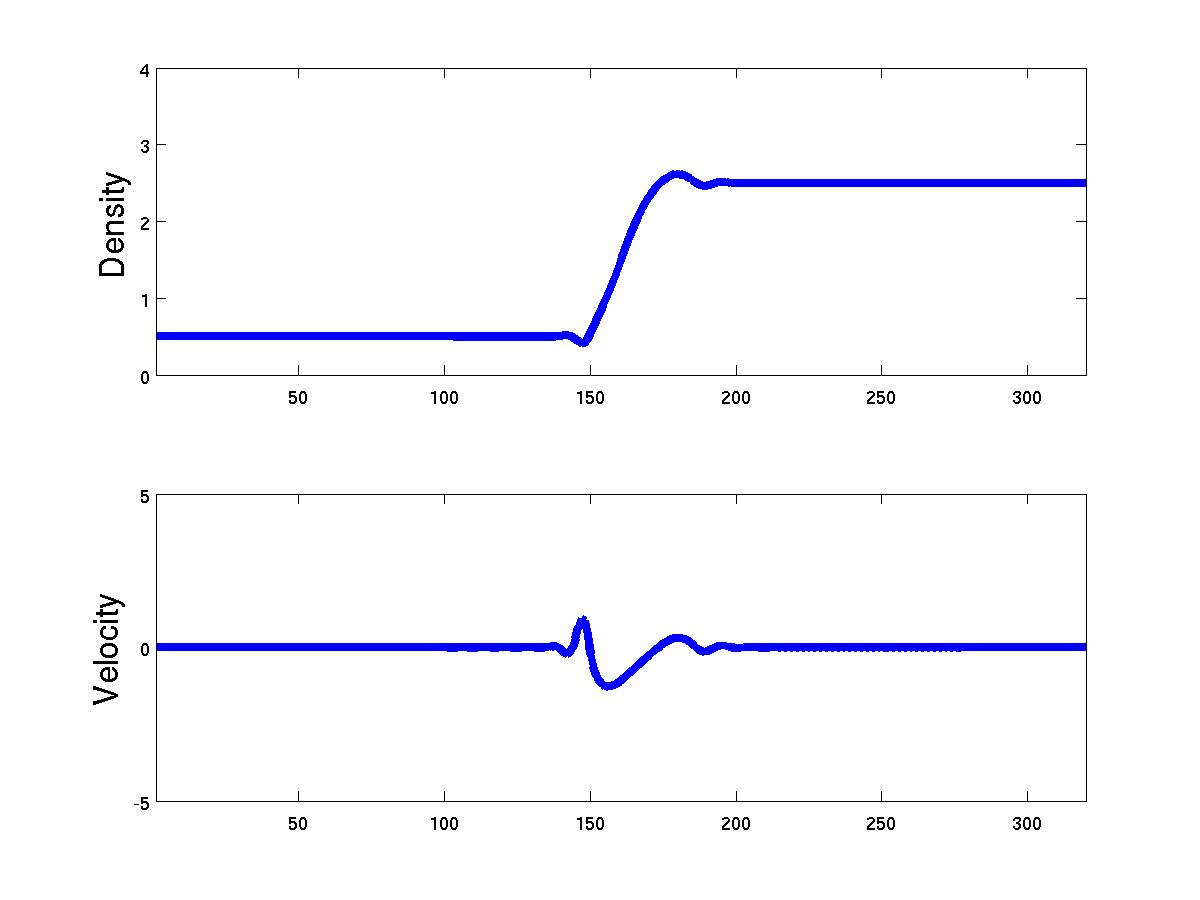}\includegraphics[width=.29\textwidth]{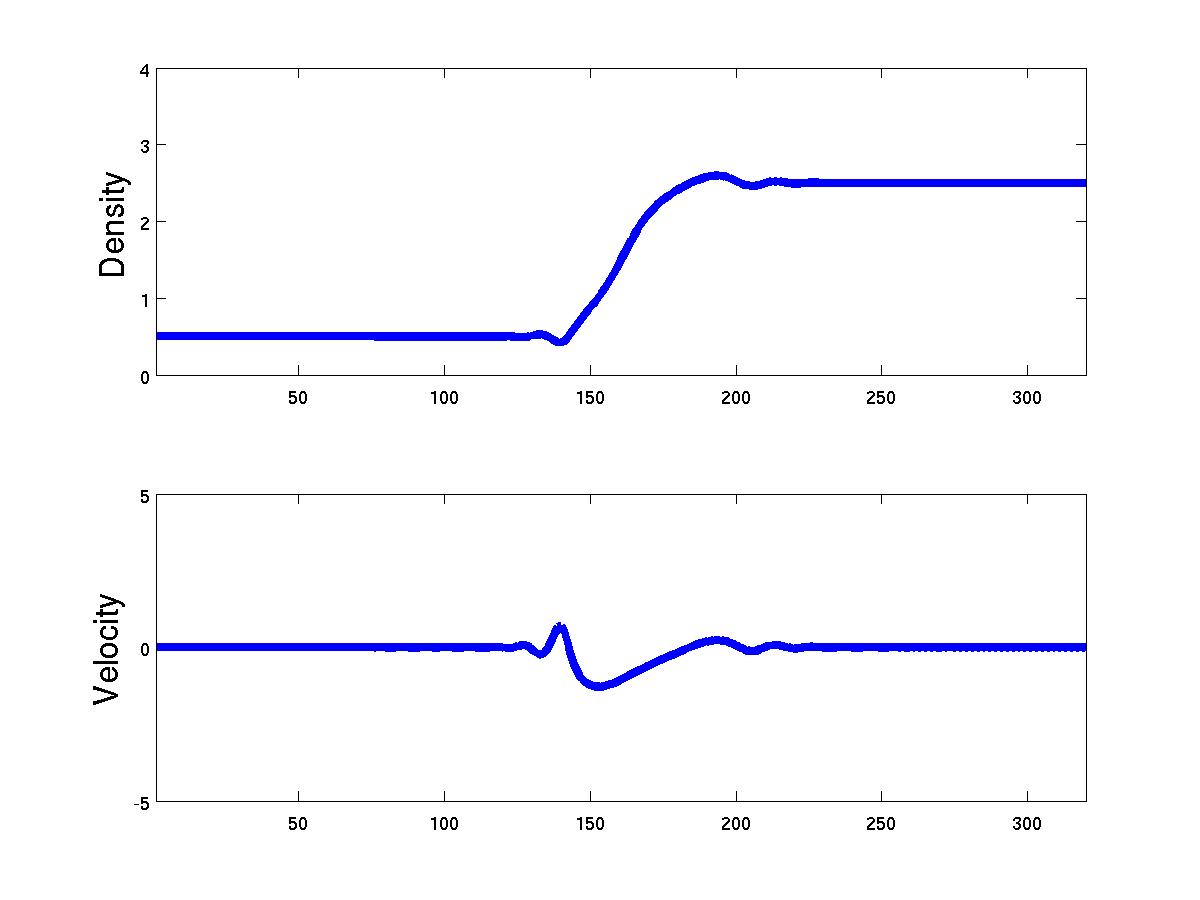}
\includegraphics[width=.4\textwidth]{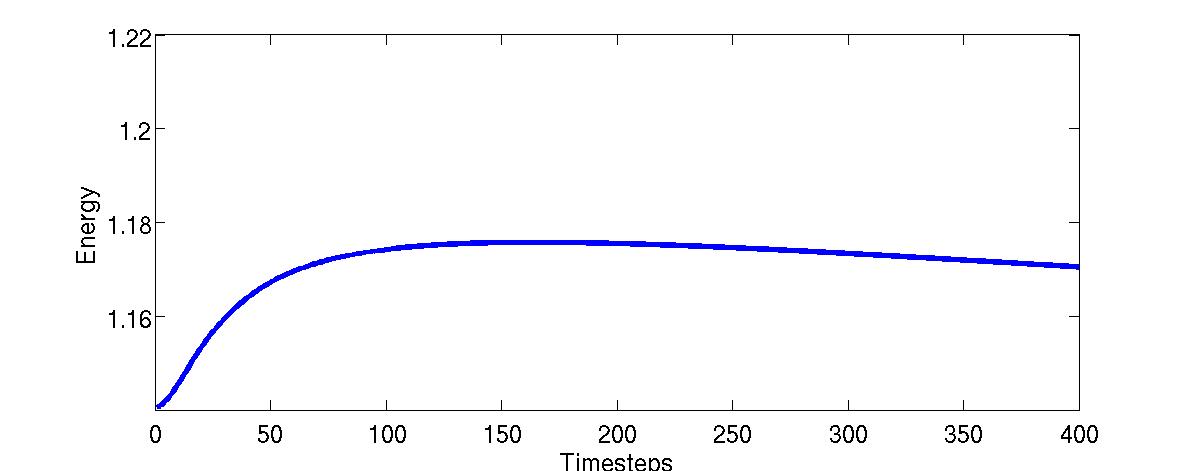}
\caption{
 Test from \S \ref{sub:lin} with $K=320$: Left: Density and velocity after 200 timesteps. Middle: Density and velocity after 400 timesteps. Right: Energy over time.
}\label{fig:lin}
\end{figure}

\bibliographystyle{plain}
\bibliography{nskbib}
\end{document}